\documentclass[final,hidelinks]{siamart251216}

\usepackage{lipsum}
\usepackage{epstopdf}
\usepackage{algorithmic}
\ifpdf
  \DeclareGraphicsExtensions{.eps,.pdf,.png,.jpg}
\else
  \DeclareGraphicsExtensions{.eps}
\fi
\usepackage{amsmath}
\usepackage{amssymb} 
\usepackage{amsfonts}
\usepackage{bm} 
\usepackage{graphicx}
\usepackage[algo2e,ruled,vlined,linesnumbered]{algorithm2e}
\usepackage{booktabs} 
\usepackage{siunitx,xfp} 
\usepackage[l3]{csvsimple}   
\usepackage{enumitem}
\usepackage{pgfplots}
\usepackage{pgfplotstable}
\pgfplotsset{compat=1.18}
\usepgfplotslibrary{groupplots}
\usepgfplotslibrary{external}
\usepackage{tikz}
\usetikzlibrary {patterns,patterns.meta} 
\usetikzlibrary{spy,calc}
\usepackage[Tol]{colorblind}
\usepackage{subcaption}
\usepackage{float}
\usepackage{multirow}

\newsiamremark{remark}{Remark}
\crefname{remark}{Remark}{Remarks}
\newsiamremark{hypothesis}{Hypothesis}
\crefname{hypothesis}{Hypothesis}{Hypotheses}
\newsiamthm{claim}{Claim}
\newsiamremark{fact}{Fact}
\crefname{fact}{Fact}{Facts}

\headers{Compact Rational Krylov for BEM Frequency Sweeping}{K. Bruyninckx, D. Huybrechs and K. Meerbergen}

\title{Compact Rational Krylov for Parametrized Systems with Application to BEM Frequency Sweeping\thanks{Submitted to the journal's Software, High-Performance Computing, and Computational Methods in Science and Engineering section August 6, 2026.\funding{This work was supported by FWO-Flanders project G088622N}}}

\author{Kobe Bruyninckx\thanks{Department of Computer Science, KU Leuven, Leuven, 3000
  (\email{kobe.bruyninckx@kuleuven.be}, \email{daan.huybrechs@kuleuven.be}, \email{karl.meerbergen@kuleuven.be}).}
\and Daan Huybrechs\footnotemark[2]
\and Karl Meerbergen\footnotemark[2]}

\usepackage{amsopn}

\ifpdf
\hypersetup{
  pdftitle={Compact Rational Krylov for Parametrized Systems with Application to BEM Frequency Sweeping},
  pdfauthor={K. Bruyninckx, D. Huybrechs, and K. Meerbergen}
}
\fi

\DeclareMathAlphabet{\mathsfit}{T1}{\sfdefault}{\mddefault}{\sldefault}
\renewcommand{\vec}[1]{\mathsf{#1}}
\newcommand{\mat}[1]{\mathsf{#1}}
\newcommand{\bvec}[1]{\bm{\mathsf{#1}}}
\newcommand{\bmat}[1]{\bm{\mathsf{#1}}}
\newcommand{\trns}[1]{{#1}^\mathsf{T}}
\newcommand{\tr}{\mathsf{T}}
\newcommand{\conj}[1]{{#1}^\mathsf{H}}
\newcommand{\cj}{\mathsf{H}}
\DeclareMathSymbol{\smin}{\mathbin}{AMSa}{"39}
\newcommand{\minone}{\smin 1} 
\newcommand{\R}{\mathbb{R}} 
\newcommand{\C}{\mathbb{C}} 
\newcommand{\bsigma}{\bm{\sigma}} 
\newcommand{\mSigma}{\mat{\Sigma}} 
\newcommand{\vx}{\vec{x}} 
\newcommand{\bvx}{\bvec{x}} 
\newcommand{\vy}{\vec{y}} 
\newcommand{\bvy}{\bvec{y}} 
\newcommand{\vz}{\vec{z}} 
\newcommand{\bvz}{\bvec{z}} 
\newcommand{\vb}{\vec{b}} 
\newcommand{\vv}{\vec{v}} 
\newcommand{\bvv}{\bvec{v}} 
\newcommand{\vf}{\vec{f}} 
\newcommand{\vr}{\vec{r}} 
\newcommand{\bvr}{\bvec{r}} 
\newcommand{\vp}{\vec{p}} 
\newcommand{\vq}{\vec{q}} 
\newcommand{\bvd}{\bvec{d}} 
\newcommand{\mP}{\mat{P}} 
\newcommand{\mC}{\mat{C}} 
\newcommand{\bbC}{\bar{\bmat{C}}} 
\newcommand{\mA}{\mat{A}} 
\newcommand{\bA}{\bmat{A}} 
\newcommand{\mB}{\mat{B}} 
\newcommand{\bB}{\bmat{B}} 
\newcommand{\bAB}[1]{\bmat{A}-#1\bmat{B}} 
\newcommand{\mM}{\mat{M}} 
\newcommand{\mN}{\mat{N}} 
\newcommand{\mMN}[1]{\mat{M}-#1\mat{N}} 
\newcommand{\mV}{\mat{V}} 
\newcommand{\bmV}{\bmat{V}} 
\newcommand{\mZ}{\mat{Z}} 
\newcommand{\bmZ}{\bmat{Z}} 
\newcommand{\mQ}{\mat{Q}} 
\newcommand{\mU}{\mat{U}} 
\newcommand{\bmU}{\bmat{U}} 
\newcommand{\mS}{\mat{S}} 
\newcommand{\bmS}{\bmat{S}} 
\newcommand{\Id}{\mat{I}} 
\newcommand{\rank}{\mathrm{rank}} 
\newcommand{\vspan}{\mathrm{span}} 
\newcommand{\diag}{\mathrm{diag}} 
\newcommand{\kron}{\otimes} 
\newcommand{\Hu}{\underline{\mat{H}}}
\newcommand{\Ku}{\underline{\mat{K}}}
\newcommand{\Tu}{\underline{\mat{T}}}
\newcommand{\red}{\mathrm{red}}

\newcommand{\hbeta}{\hat{\beta}}
\newcommand{\setmu}{S}
\newcommand{\cU}{\mathcal{U}} 
\newcommand{\cL}{\mathcal{L}} 
\newcommand{\cP}{\mathcal{P}} 
\newcommand{\mPu}{\underline{\mP}}
\newcommand{\mCu}{\underline{\mC}}
\newcommand{\vxu}{\underline{\vx}}
\newcommand{\vbu}{\underline{\vb}}
\newcommand{\vru}{\underline{\vr}}

\newcommand{\vqu}{\underline{\vq}}
\newcommand{\px}{\mathbf{x}} 
\newcommand{\py}{\mathbf{y}} 
\newcommand{\Gk}{G_\kappa} 
\newcommand{\Hm}{\mathcal{H}}
\newcommand{\restr}[2]{{#1}_{|#2}}
\newcommand{\bt}{t}
\newcommand{\bs}{s}
\DeclareMathOperator*{\argmin}{arg\,min}
\DeclareMathOperator*{\argmax}{arg\,max}

\def\CC{{C\nolinebreak[4]\hspace{-.05em}\raisebox{.4ex}{\tiny\bf ++}}}

\newcolumntype{N}[1]{S[table-format=#1]}

\begin{document}
\maketitle

\begin{abstract}
In parametrized linear systems $\mP(\mu)\vx=\vb$ the system matrix $\mP$ depends nonlinearly on a parameter $\mu$ and solutions are sought for many values of this parameter. We show that the compact rational Krylov (CORK) framework, originally introduced to solve nonlinear eigenvalue problems, can be used to efficiently produce approximate solutions to such a system for many values of the parameter at once. In this approach, the parametrized system is first linearized, resulting in a large shifted linear system $(\bA-\mu\bB)\bvy=\bvd$. We formulate a left- and right-preconditioned rational Krylov GMRES method for shifted linear systems. In the setting of parametrized linear systems, these can exploit the structure in the linearization, and in combination with the CORK framework, computational and memory complexity mainly depend on the problem size, less the degree of the linearization. Additionally, we show how to incorporate a right-hand side $\vb(\mu)$ that also depends on the parameter, how to choose the shifts to steer convergence and how to allow for inexact solves at these shifts throughout the iterations. As an application we consider the ‘frequency sweeping’ of Helmholtz scattering problems through the Boundary Element Method (BEM), enabled via an efficient representation of the dense but data-sparse wavenumber-dependent system matrix.
\end{abstract}

\begin{keywords}
rational Krylov, parametrized linear systems, companion linearization, shifted linear systems, inexact Krylov, hierarchical matrices, boundary element method, Helmholtz equation

\end{keywords}

\begin{MSCcodes}
65F10, 65F55, 65N38, 35J05
\end{MSCcodes} 

\section{Introduction}\label{sec:1_intro}

The motivation of this work arises from the field of integral equations. The Boundary Element Method (BEM) is a popular numerical technique for solving scattering problems in the frequency domain, based on an integral equation reformulation of the governing partial differential equation, such as the Helmholtz equation. Among many advantages of this approach, the resulting BEM system matrix is generally dense and nonlinearly dependent on the frequency. The dense matrix can still be approximated in a data-sparse way, using a range of algebraic and analytic techniques~\cite{GreengardRokhlin1987,Hackbusch1999,HackbuschBorm2002,XiaChandrasekaran2010,HoYing2016,MindenEtAl2017,JiangGreengard2025,BruyninckxEtAl2025} and the resulting systems can be solved efficiently with iterative or direct solvers~\cite{Bebendorf2005,ChaiJiao2010,ChaiJiao2011,HoYing2016,CoulierEtAl2017,SushnikovaEtAl2023}.

In a `frequency sweep', the solution to a scattering problem is sought over a range of frequencies. While the above techniques speed up solution at a single frequency, repeating such procedures for all frequencies of interest may still be intractable. Earlier research on frequency sweeping has involved both the Finite Element Method (FEM) and BEM, see \cite{Marburg2024} for a recent survey. Frequency sweeping requires tackling both the efficient construction and representation of the nonlinear frequency-dependent system, e.g.\ \cite{GopalMartinsson2023,Dirckx2022}, as well as efficiently resolving the solutions over the frequency range, e.g. \cite{Panagiotopoulos2023,BaydounEtAl2020,BaydounEtAl2021}.

The wavenumber-dependent system for Helmholtz BEM leads to a parametrized linear system, given in general by
\begin{equation}\label{eq:parametrized_linear_system}
    \mP(\mu)\vx(\mu) = \vb(\mu)
\end{equation}
with $\mu\in\setmu\subseteq\C$ and $\mP : S \to \C^{n\times n}$, where $\vx(\mu)$ is sought for many values of $\mu$. 
The study of parametrized linear systems is rich, see \cite{GrammontEtAl2011,SoodhalterEtAl2020,KresnnerTobler2011,Correnty2024} and references therein.

Using companion linearization techniques, we formulate the BEM system as a set of shifted linear systems, linear in the parameter $\mu$, with a matrix pencil $(\bA,\bB)$ as
\begin{equation}\label{eq:shifted_linear_system}
    (\bA-\mu \bB)\bvy = \bvd.
\end{equation}
This involves the rational approximation of $\mP(\mu)$. The linearized system can be solved for many different values of $\mu$ simultaneously using Krylov subspace methods \cite[\S14.1]{SimonciniSzyld2007}, but it has (much) larger dimensions than the original one. We exploit the structure of the companion pencil and take inspiration from the compact rational Krylov (CORK) framework, originally formulated for nonlinear eigenvalue problems~\cite{VanBeeumen2015}, to efficiently apply rational Krylov.

\subsection{Related work} Relevant related work for parametrized systems was described in~\cite{JarlebringCorrenty2022,Correnty2024}, which combined a linearization using Taylor series expansions with a GMRES solver, and in~\cite{Correnty2023}, using Chebyshev interpolation of $\mP(\mu)$ in combination with BiCG. However, rational Krylov was not used in these references. In comparison, the CORK framework enables general companion linearizations, such as those associated with rational functions, and owing to rational Krylov it also allows for improved convergence of the solutions through the use of multiple shifts.

Rational Krylov for shifted systems was investigated in \cite{DaasPalitta2026}, leading to a scheme similar to the one we propose using left preconditioning, but for the special case $\bB=\Id$ and without initial shift. Rational Krylov for shifted systems with right preconditioning is equivalent to the so-called `flexible Krylov' techniques of \cite{GuEtAl2007,SaibabaEtAl2013}. In both settings, we leverage the linearization pencil structure for faster computations, and incorporate the CORK framework to reduce memory footprint.

\subsection{Results} The contributions of this article are the following. We:
\begin{itemize}[noitemsep,topsep=0pt]
    \item demonstrate how the CORK framework can be used to solve parametrized linear systems with a general pencil $(\bA,\bB)$ for a range of the parameter $\mu$;
    \item introduce a right-preconditioned variant of CORK;
    \item allow parameter dependence of the right-hand side of~\eqref{eq:parametrized_linear_system};
    \item incorporate inexact Krylov theory \cite{SimonciniSzyld2003,GiraudEtAl2007} into CORK;
    \item apply the framework, combined with the compact wavenumber-dependent BEM representation from \cite{Dirckx2022,Dirckx2024,Bruyninckx2026}, to efficiently perform frequency sweeps.
\end{itemize}
The frequency sweeping method is valid for a continuous frequency range, based on solves at a small number of frequency values. They correspond to shifts of the rational Krylov method, and can be greedily chosen very effectively based on a cheaply available estimate of the residual across the range. This is a highly attractive feature of the scheme. Afterwards, each additional solve requires only the solution of a very small linear system, whose size relates to that of the constructed Krylov subspace. We show that the method is very performant for small to moderately large frequencies.

\subsection{Outline} We summarize rational Krylov GMRES for shifted linear systems in~\S\ref{sec:2_rational-krylov}, including a variant involving right preconditioning and a greedy shift selection procedure. We introduce compact rational Krylov (CORK) for parameterized linear systems with varying right-hand sides in~\S\ref{sec:3_cork}. We integrate inexact Krylov theory in~\S\ref{sec:4_inexact-rk}. The methodology is applied to BEM frequency sweeping problems in~\S\ref{sec:5_bem} and we further illustrate the approach numerically in~\S\ref{sec:6_numexp}.

\section{Rational Krylov GMRES for shifted linear systems}\label{sec:2_rational-krylov}
The rational Krylov method was developed in \cite{Ruhe1984,Ruhe1998} and generalizes shift-and-invert Arnoldi for the solution of (generalized) eigenvalue problems. This is achieved by allowing multiple different shifts throughout the iterative process.

Rational Krylov for matrix pencils works as follows. Given a starting vector $\bvv_1$, a matrix pencil $(\bA,\bB)$ and a set of shifts $\bsigma_m=(\sigma_i)_{i=1}^{m}$, the algorithm constructs the rational Krylov subspace
\begin{equation}\label{eq:ratkrylov-subspace}
    \mathcal{K}_{m+1} = \vspan{ \left\{\bvv_1,(\bA-\sigma_1\bB)^{\minone}\bB\bvv_1,\dots,\prod_{i=1}^{m}\left((\bA-\sigma_i\bB)^{\minone}\bB\right)\bvv_1 \right\} }
\end{equation}
and builds an orthogonal basis $\bmV_{m+1}:=[\bvv_1\cdots\bvv_{m+1}]$ satisfying the recurrence relation
\begin{equation}\label{eq:rat-krylov-reccurence}
    \bA \bmV_{m+1} \Hu_m = \bB \bmV_{m+1} \Ku_m
\end{equation}
with $\Hu_m,\Ku_m\in\C^{(m+1)\times m}$ both upper Hessenberg. Letting $\mSigma_m:=\diag{(\bm{\sigma}_m)}$, $\Ku_m$ is related to the shifts as $\Ku_m=\Hu_m\mSigma_m+\Tu_m$, in which upper triangular matrix $\Tu_m$ contains the selected continuation combinations, see \cite[\S3.2.3]{Ruhe1998} and \cite[\S1.3.2]{VanBeeumen2015b}. 

Rational Krylov methods are typically applied to eigenvalue problems. Approximations to eigenpairs of $(\bA,\bB)$ can be found from the eigenpairs of the pencil $(\mat{K}_m,\mat{H}_m)$. The approximate eigenvectors live in the subspace spanned by $\bmV_{m+1}\Hu_m$.

Here, instead, we aim to solve shifted linear systems~\eqref{eq:shifted_linear_system} with pencil $(\bA,\bB)$. When trying to find eigenvalues, rational Krylov chooses the shifts $\bsigma_m$ in the region of interest to steer convergence. Similarly, we can choose the shifts close to values of $\mu$ for which we want to solve. In this way, we hope the rational Krylov subspace contains good approximations to the solutions for all these values of $\mu$.

We discuss the use of rational Krylov as an algorithm to solve shifted systems. This leads first to left-preconditioned rational Krylov GMRES (LRK-GMRES).

\subsection{Left-preconditioned RK-GMRES}\label{ssec:left-rk-gmres}

A shift-and-invert strategy can be combined with GMRES to solve shifted linear systems. One chooses a shift $\sigma$ in the range of interest for $\mu$, with associated preconditioner $(\bA-\sigma\bB)^{\minone}$. The proposition below, due to~\cite{Meerbergen2003}, formalizes how to exploit the shift-invariance of Krylov subspaces to solve the shifted system for any value of $\mu$. The solution is found from a small least-squares system, whose accuracy relates to the preconditioned residual.
\begin{proposition}[Shift-and-invert GMRES]
    Assume $m$ steps of shift-and-invert Arnoldi have been performed on the matrix pencil $(\bA,\bB)$ with shift $\sigma$, i.e.\ Arnoldi on $(\bA-\sigma\bB)^{\minone}\bB$, with starting vector $\bvv_1=\hat{\bvd}/\hbeta$, $\hat{\bvd}=(\bA-\sigma \bB)^{\minone}\bvd$, $\hbeta=\|\hat{\bvd}\|$ such that 
    \begin{equation}\label{eq:shift-invert-recurrence}
        (\bA-\sigma \bB)^{\minone}\bB \bmV_m = \bmV_{m+1}\Hu_m
    \end{equation}
    holds with $\bmV_m=[\bvv_1 \,\cdots\, \bvv_m]$ orthogonal and $\Hu_m$ upper Hessenberg. Let $\vy^\mathrm{red}_m$ solve
    \begin{equation}\label{eq:shift_and_invert_y}
\vy^\red_m = \argmin_\vy \|\hbeta\vec{e}_1-\Hu_m^\sigma(\mu)\vy\|, \quad \mbox{with} \quad \Hu_m^\sigma(\mu) :=\left[\begin{array}{c}
        \Id_m \\
        0
    \end{array}\right] + (\sigma-\mu)\Hu_m,
    \end{equation}
    a least-squares problem of size $(m+1)\times m$. $\bvy_m(\mu)=\bmV_m\vy^\mathrm{red}_m(\mu)$ minimizes the norm of the preconditioned residual
    \begin{equation}\label{eq:preconditioned_residual_shift_and_invert}
    \bvec{r}_m^{\sigma}(\mu) := (\bA-\sigma \bB)^{\minone}\left(\bvd - (\bA-\mu \bB)\bvy_m\right)
    \end{equation}
    of the shifted linear system~\cref{eq:shifted_linear_system} with norm $
\|\bvec{r}_m^{\sigma}\|=\|\hbeta\vec{e}_1-\Hu_m^\sigma(\mu)\vy^\red_m\|$.
\end{proposition}
\begin{proof}
Following \cite{Meerbergen2003}, we have that $
(\bA-\sigma \bB)^{\minone}(\bA-\mu \bB)\bmV_m = \bmV_{m+1}\Hu_m^\sigma(\mu)$. Choosing an approximate solution of the form $\bvy_m(\mu)=\bmV_m\vy^\mathrm{red}_m(\mu)$ gives $
    \bvec{r}_m^{\sigma} = \bmV_{m+1}(\hbeta\vec{e}_1\allowbreak-\Hu_m^\sigma(\mu)\vy^\red_m)
$ and the result follows.
\end{proof}

In rational Krylov, instead of applying $(\bA-\sigma \bB)^{\minone}\bB$ in each iteration for a fixed $\sigma$, the shift varies. This results in recurrence relation \cref{eq:rat-krylov-reccurence} instead of \cref{eq:shift-invert-recurrence}. The following proposition can be seen to generalize the rational Krylov subspace method for shifted linear systems from \cite{DaasPalitta2026}; the latter is obtained by taking $\bB=\Id$ and $\sigma_0=\infty$.

\begin{proposition}[Left-preconditioned rational Krylov GMRES]\label{prop:lrk-gmres}
    Assume $m$ steps of rational Krylov have been performed on the matrix pencil $(\bA,\bB)$ with shifts $\bsigma_m=(\sigma_i)_{i=1}^m$ and starting vector $\bvv_1=\hat{\bvd}/\hbeta$, $\hat{\bvd}=(\bA-\sigma_0 \bB)^{\minone}\bvd$, $\hbeta=\|\hat{\bvd}\|$, i.e., the rational Krylov subspace~\cref{eq:ratkrylov-subspace} is constructed such that \cref{eq:rat-krylov-reccurence}
    holds. Let $\vy^\mathrm{red}_m$ solve
    \begin{equation}\label{eq:rational_krylov_y}
        \vy^\red_m = \argmin_\vy \|\hbeta\vec{e}_1-(\Ku_m - \mu \Hu_m)\vy\|,
    \end{equation}
    a least-squares problem of size $(m+1)\times m$. $\bvy_m(\mu)=\bmV_{m+1}(\Ku_m - \sigma_0 \Hu_m)\vy^\mathrm{red}_m(\mu)$ minimizes the norm of the preconditioned residual
    \begin{equation}\label{eq:preconditioned_residual_rational_krylov}
    \bvec{r}_m^{\sigma}(\mu) := (\bA-\sigma_0 \bB)^{\minone}\left(\bvd - (\bA-\mu \bB)\bvy_m\right)
    \end{equation}
    of the shifted linear system~\cref{eq:shifted_linear_system} with norm $\|\bvec{r}_m^{\sigma}\|=\|\hbeta\vec{e}_1-(\Ku_m - \mu \Hu_m)\vy^\red_m\|$.
\end{proposition}
\begin{proof}
Adding $-\sigma_0 \bB\bmV_{m+1}\Hu_m$ to~\cref{eq:rat-krylov-reccurence} and applying $(\bA-\sigma_0 \bB)^{\minone}$ yields
\begin{equation*}
    (\bA -\sigma_0 \bB)^{\minone} \bB \bmV_{m+1} (\Ku_m - \sigma_0 \Hu_m) = \bmV_{m+1} \Hu_m.
\end{equation*}
Scaling by $(\sigma_0-\mu)$ and adding $\bmV_{m+1} (\Ku_m - \sigma_0 \Hu_m)$ gives
\begin{equation}\label{eq:rat-krylov-reccurence-2}
    (\bA-\sigma_0 \bB)^{\minone}(\bA-\mu \bB) \bmV_{m+1} (\Ku_m - \sigma_0 \Hu_m) = \bmV_{m+1} (\Ku_m - \mu \Hu_m).
\end{equation}
Using an approximate solution of the form $\bvy_m(\mu)=\bmV_{m+1}(\Ku_m - \sigma_0 \Hu_m)\vy^\mathrm{red}_m(\mu)$, the preconditioned residual equals
\begin{align*}
    \bvec{r}_m^{\sigma} &= (\bA-\sigma_0 \bB)^{\minone}\bvd - (\bA-\sigma_0 \bB)^{\minone}(\bA-\mu \bB)\bmV_{m+1}(\Ku_m - \sigma_0 \Hu_m)\vy^\red_m \\
        &= (\bA-\sigma_0 \bB)^{\minone}\bvd - \bmV_{m+1} (\Ku_m - \mu \Hu_m)\vy^\red_m.
\end{align*}
With $(\bA-\sigma_0 \bB)^{\minone}\bvd=\hbeta\bvv_1$ we obtain $\bvec{r}_m^{\sigma} = \bmV_{m+1} \left(\hbeta\vec{e}_1-(\Ku_m - \mu \Hu_m)\vy^\red_m\right)$.
\end{proof}

An initial shift $\sigma_0$ remains in the scheme and appears in the preconditioned residual. Hence, its choice plays a role in the convergence of the scheme.

\begin{remark}
We shall not work out the details here, but like to point out that technically one could use a trial space for the solution spanned by $\bmV_{m+1}(\Ku_m - \sigma_i \Hu_m)$ with $0 \leq i \leq m$, i.e., also using one of the other shifts. One can show that the corresponding solution at that shift is deflated out of the trial space. See \cite{Bruyninckx2026}.
\end{remark}

\subsection{Right-preconditioned RK-GMRES}\label{ssec:right-rk-gmres}

The question arises whether the choice of an initial shift can be avoided. A shift-invariant Krylov subspace can also be obtained using right preconditioning instead. As it turns out, this setting is exactly the same as that of `flexible GMRES for shifted systems'~\cite{SaibabaEtAl2013}, see also the flexible GMRES algorithm of \cite{GuEtAl2007}. For this reason, we shall be brief in the description here.\footnote{We use the convention of defining shifted systems by subtracting the second matrix instead of adding. This slightly changes the equations compared to \cite{SaibabaEtAl2013}.}

Right preconditioning for the solution of $(\bA-\mu \bB)\bvy = \bvd$ means
\begin{equation*}
    (\bA-\mu \bB)(\bA-\sigma \bB)^{\minone}\bar{\bvy} = \bvd, \quad \mbox{with} \quad \bvy = (\bA-\sigma \bB)^{\minone}\bar{\bvy}.
\end{equation*}
This is equivalent with $\left(\Id + (\sigma-\mu)\bB(\bA-\sigma \bB)^{\minone}\right)\bar{\bvy} = \bvd$ and suggests applying Arnoldi on $\bB(\bA-\sigma \bB)^{\minone}$ instead of $(\bA-\sigma \bB)^{\minone}\bB$. With multiple shifts, this leads to
\begin{equation}\label{eq:ratkrylov-subspace-right}
    \mathcal{K}_{m+1}^{\textrm{right}} = \vspan{ \left\{\bvv_1,\bB(\bA-\sigma_1\bB)^{\minone}\bvv_1,\dots,\prod_{i=1}^{m}\left(\bB(\bA-\sigma_i\bB)^{\minone}\right)\bvv_1 \right\} }.
\end{equation}

This time around, we can use the starting vector $\bvv_1:=\bvd/\|\bvd\|$. In this context, a matrix $\bmZ_m$ arises due to the presence of the flexible right preconditioner. It satisfies
\begin{equation*}
    \bB\bmZ_m = \bmV_{m+1}\Hu_m, \quad \mbox{and} \quad \bA\bmZ_m - \bB\bmZ_m\mSigma_m = \bmV_m\mat{T}_m.
\end{equation*}
Multiplying the first equation on the right by $\mSigma_m-\mu \Id_m$ shows that
\begin{equation}\label{eq:rrkgmres-shifted-recurrence}
    (\bA-\mu \bB)\bmZ_m = \bmV_{m+1}\left(\Tu_m + \Hu_m(\mSigma_m-\mu \Id_m)\right) = \bmV_{m+1}(\Ku_m - \mu \Hu_m).
\end{equation}
Thus, the approximate solution lies in the column space of $\bmZ_m$. It is obtained as $\bvy_m(\mu)=\bmZ_m\vy^\red_m(\mu)$, where $\vy^\mathrm{red}_m$ solves \cref{eq:rational_krylov_y} as before.

One advantage of right preconditioning is that the residual norm of the projected least-squares system, which is minimized, is equal to that of the original system,
\begin{equation}\label{eq:rrk_residual}
 \|\bvec{r}_m\| = \|\bvd - (\bA-\mu \bB)\bvy_m\| = \|\beta\vec{e}_1-(\Ku_m-\mu\Hu_m)\vy_m^\red\|,
\end{equation}
$\beta=\|\bvd\|$, whereas with left preconditioning one obtains a preconditioned residual norm. On the other hand, the matrix $\bmZ_m$ needs to be stored in addition to $\bmV_{m+1}$.

\subsection{Greedy selection of shifts in rational Krylov}\label{ssec:greedy-rk} 

The distribution of the shifts in rational Krylov relative to the values of $\mu$ for which one wants to solve the system determines the effectiveness of the approach. It can generally be expected that the introduction of a shift $\sigma_k$ at iteration $k$ is most beneficial for nearby values of $\mu$. Logically, one may choose the next shift where the error is currently large.

One of the compelling features of rational Krylov is that a cheap estimate of the residual norm is available for all values of $\mu$, precisely because one can solve a small least-squares system. That system involves an upper Hessenberg matrix $\Ku_k-\mu\Hu_k$, which can be reduced to upper triangular form in $\mathcal{O}(k^2)$ operations. A practical strategy is to choose a finite subset $\hat{S} \subset S$ of values for $\mu$ in its range of interest. A straightforward greedy approach consists in choosing the next shift as the maximizer of the residual norm over $\hat{S}$. For right and left-preconditioned RK-GMRES this means
\begin{equation*}
\sigma_{k+1} = \argmax_{\mu\in \hat{S}}{\|\bvec{r}_k(\mu)\|}, \quad \mbox{or} \quad \sigma_{k+1} = \argmax_{\mu\in \hat{S}}{\|\bvec{r}_k^{\sigma}(\mu)\|}, \quad\text{respectively.}
\end{equation*}
All that remains is to select an initial shift $\sigma_1$ and, for LRK-GMRES, also $\sigma_0$.

Such a greedy scheme was recently investigated in \cite{DaasPalitta2026}, in which shifts are referred to as `poles'. It was compared with the so-called ADM shift selection strategy of \cite{CasulliRobol2024}, which is based on a representation of the residual norm $\|\bvec{r}_k(\mu)\|$ as a rational function in $\mu$. It was found that the simple greedy strategy performs significantly better.

Our greedy scheme differs in two ways from that of \cite{DaasPalitta2026}. First, in the setting of parametrized linear systems we are primarily interested in the residual of the parametrized system, not of its linearization. The second difference is that we use inexact iterative solvers at each shift, which means that the computed residuals no longer match the true ones. These are the topics of the next two sections.

\section{Compact rational Krylov for parametrized linear systems}\label{sec:3_cork}
We move from rational Krylov as a means to solve families of shifted linear systems to its application for the solution of parametrized linear systems~\cref{eq:parametrized_linear_system}. The exposition is based on the original paper \cite{VanBeeumen2015}, in which compact rational Krylov (CORK) was proposed for nonlinear eigenvalue problems. As mentioned before, other methods have been described using linearization techniques for the solution of parametrized linear systems, see \cite{JarlebringCorrenty2022,Correnty2023,Correnty2024}. In comparison to those, CORK accommodates general companion linearizations rather than specific choices. It also supports the use of multiple shifts to steer convergence, permitting a greedy selection strategy.

\subsection{Linearization and left-preconditioned CORK-GMRES}\label{ssec:left-cork}

For CORK, it is assumed that $\mP(\mu)$ takes the form
\begin{equation}\label{eq:matrix-function-add}
    \mP(\mu) = \sum_{i=1}^d (\mA_i-\mu \mB_i)f_i(\mu), \qquad \mA_i,\mB_i\in\C^{n\times n}, \quad f_i: \C \to \C.
\end{equation}
where (rational or polynomial) basis functions $\vf(\mu)=\trns{[f_1(\mu) \,\cdots\, f_d(\mu)]}$ satisfy the linear relationship
\begin{equation}\label{eq:linear-relation}
    (\mMN{\mu})\vf(\mu) = \vec{0}.
\end{equation}
with $\mM,\mN\in\C^{(d-1)\times d}$ and $\rank{(\mMN{\mu})} = d-1$. We refer to \cite[Table 1]{VanBeeumen2015} for several examples. 

In practice, $\mP$ is often given in the form $\mP(\mu) = \sum_{i=1}^s \mat{C}_i g_i(\mu)$ from which the representation \cref{eq:matrix-function-add} can be derived. This may involve approximating $\{g_i\}_{i=1}^s$ by polynomials or rational functions, e.g.,\ using AAA~\cite{LietaertEtAl2022,GuttelEtAl2024}. Matrix operations involving $\{\mA_i\}_{i=1}^d$ and $\{\mB_i\}_{i=1}^d$ can also be expressed in terms of $\{\mat{C}_i\}_{i=1}^s$ instead. This is typically computationally advantageous because application-specific structure can be exploited or because $s\ll d$.

Linearization of $\mP$ leads to a larger matrix pencil $\bAB{\mu} \in \C^{dn\times dn}$,
with
\begin{equation}\label{eq:linearization}
    \bA = \left[\begin{array}{cccc}
        \mA_1 & \mA_2 & \cdots & \mA_d \\
        \noalign{\vskip 2pt} \hline \noalign{\vskip 2pt}
        \multicolumn{4}{c}{\mM \kron \Id_n }
    \end{array}\right],\quad\quad
    \bB = \left[\begin{array}{cccc}
        \mB_1 & \mB_2 & \cdots & \mB_d \\
        \noalign{\vskip 2pt} \hline \noalign{\vskip 2pt}
        \multicolumn{4}{c}{\mN \kron \Id_n }
    \end{array}\right].
\end{equation}
The following theorem shows the connection between the shifted linear system $\bA-\mu\bB$ and the parametrized linear system $\mP(\mu)\vx=\vb$ (with constant right-hand side). A more general, but also more abstract, statement was given in \cite[Theorem 4.1]{GrammontEtAl2011}. 
\begin{theorem}[Linearization]\label{thm:linearization}
    Let $\mP(\mu)$ be given by \cref{eq:matrix-function-add} with \cref{eq:linear-relation}, and define $\bA - \mu \bB$ by \cref{eq:linearization}. Then, $\vx\in\C^n$ is the unique solution to $\mP(\mu)\vx=\vb$, $\vb\in\C^n$, if and only if $\bvy$ is the unique solution to $(\bA-\mu\bB)\bvy=\bvd$ of the form
    \begin{equation*}
        \bvy = \vf(\mu) \kron \vx \in \C^{dn}
    \end{equation*}
    with $\bvd = [\trns{\vb} \; \vec{0}]^\tr = \vec{e}_1 \kron \vb$.
\end{theorem}
\begin{proof}
The last $(d-1)$ block rows of size $n$ of $\bA-\mu\bB$ provide the underdetermined system $\left((\mMN{\mu})\kron \Id_n\right) \bvy = \vec{0}$. Defining $\mat{Y}:=[\vy^{[1]} \; \vy^{[2]} \,\cdots\, \vy^{[d]}]\in\C^{n\times d}$, consisting of blocks $\vy^{[i]}\in\C^n$ of $\bvy$, this is equivalent to 
\begin{equation*}
    (\mMN{\mu})\trns{\mat{Y}} = \vec{0}_{(d-1)\times n}
\end{equation*}
This enforces the linear relationship \cref{eq:linear-relation} onto $\trns{\mat{Y}}$ s.t.\ $\trns{\mat{Y}}=\vf(\mu)\trns{\vz}$, i.e., $\bvy=\vf(\mu)\kron \vz$, for some $\vz\in\C^n$. Inserting this form into the first block row of $\bA-\mu\bB$ gives
\begin{equation*}
    \left[\begin{array}{ccc}
        \mA_1-\mu \mB_1 & \cdots & \mA_d-\mu \mB_d
    \end{array}\right]\left[\begin{array}{c}
        \vf_1(\mu)\vz \\
        \vdots \\
        \vf_d(\mu)\vz
    \end{array}\right]=\left(\sum_{i=1}^d (\mA_i-\mu \mB_i)\vf_i(\mu)\right)\vz = \mP(\mu)\vz = \vb
\end{equation*}
and thus $\vz=\vx$, the unique solution to the original parametrized linear system. The converse is easily seen by multiplying $\bA-\mu\bB$ with $\vf(\mu)\kron\vx$, with $\vx$ the unique solution to $\mP(\mu)\vx=\vb$.
\end{proof}
\Cref{thm:linearization} reveals that by using the pencil $(\bA,\bB)$, together with the appropriate right-hand side, we can employ methods for shifted linear systems to solve the parametrized linear systems \cref{eq:parametrized_linear_system}. We shall consider variable right-hand sides in \S\ref{ssec:cork-rhs}.

For nonlinear eigenvalue problems~\cite{VanBeeumen2015} CORK finds eigenpairs by applying the rational Krylov method \cite{Ruhe1984,Ruhe1998} on the companion pencil $(\bA,\bB)$. The CORK framework exploits the Kronecker structure present in the large pencil to efficiently employ Krylov methods. See \cite{VanBeeumen2015,Bruyninckx2026} for more details. Because CORK uses the rational Krylov subspace \cref{eq:ratkrylov-subspace}, it corresponds to LRK-GMRES from \S\ref{ssec:left-rk-gmres}. We will further refer to it as left-preconditioned CORK-GMRES (LCORK-GMRES).

\subsection{Right-preconditioned CORK-GMRES (RCORK-GMRES)}\label{ssec:right-cork}
In this setting we mimick the strategy of CORK and aim to exploit the Kronecker structure in the companion pencil for right preconditioning as well. For the sake of conciseness, the structure of this subsection parallels that of section 4 in \cite{VanBeeumen2015}. We also refer to \cite{Bruyninckx2026}, which includes an explicit overview of steps that are only implicit in literature as well as some material omitted here for brevity.

We start with restating Theorem 2.3 from \cite{VanBeeumen2015}.
\begin{theorem}[Block ULP decomposition, {\cite[Theorem 2.3]{VanBeeumen2015}}]\label{thm:ULP}
Let $(\bA,\bB)$ be the linearization \cref{eq:linearization}, satisfying \cref{eq:linear-relation}. For every $\mu\in\C$ there exists a permutation $\cP\in\C^{d\times d}$ such that matrix $\mM_2-\mu \mN_2\in\C^{(d-1)\times(d-1)}$ is invertible with
\begin{equation*}
    \mM =: \left[\begin{array}{cc}
        \vec{m}_1 & \mM_2
    \end{array}\right]\cP,\quad\quad
    \mN =: \left[\begin{array}{cc}
        \vec{n}_1 & \mN_2
    \end{array}\right]\cP.
\end{equation*}
Moreover, pencil $\bmat{L}(\mu):=\bA - \mu \bB$ can be factorized as follows:
\begin{equation*}
    \bmat{L}(\mu) = \cU\cL(\cP\kron I_n)
\end{equation*}
where
\begin{align*}
    \cL &= \left[\begin{array}{cc}
        \mP(\mu) & \mat{0} \\
        (\vec{m}_1-\mu \vec{n}_1)\kron \Id_n & (\mM_2-\mu \mN_2)\kron \Id_n 
    \end{array}\right],\\
    \cU &= \left[\begin{array}{cc}
        \alpha^{\minone} \Id_n & (\bar{\bA}_2-\mu\bar{\bB}_2)\left((\mM_2-\mu \mN_2)^{\minone}\kron \Id_n\right)  \\
        \mat{0} & \Id_{(d-1)n}
    \end{array}\right]
\end{align*}
with the scalar $\alpha=\trns{\vec{e}}_1\cP \vf(\mu)\ne0$ and
\begin{align*}
    \bar{\bA} := \left[\begin{array}{cccc}
        \mA_1 & \mA_2 & \cdots & \mA_d 
    \end{array}\right] &=: 
        \left[\begin{array}{cc}
            \bar{\mA}_1 & \bar{\bA}_2
        \end{array}\right](\cP\kron I_n),\\
    \bar{\bB} := \left[\begin{array}{cccc}
        \mB_1 & \mB_2 & \cdots & \mB_d 
    \end{array}\right] &=:
        \left[\begin{array}{cc}
            \bar{\mB}_1 & \bar{\bB}_2
        \end{array}\right](\cP\kron I_n).
\end{align*}
\end{theorem}

In addition to the pencil having block structure, so do the matrices $\bmV_{m+1}$ and, in the right-preconditioned case, $\bmZ_{m}$. With blocks $\mV_m^{[i]}\in\C^{n\times m}$ and $\vv_{m+1}^{[i]}\in\C^n$, for $i=1,\dots,d$, we write $\bmV_{m+1}\in\C^{dn\times(m+1)}$ as
\begin{equation*}
    \bvec{V}_{m+1} = \left[\begin{array}{cc}
        \bm{\mathsf{V}}_m & \bm{\mathsf{v}}_{m+1}
    \end{array}\right] = 
        \left[\begin{array}{cc}
            \mathsf{V}_m^{[1]} & \mathsf{v}_{m+1}^{[1]} \\
            \vdots & \vdots \\
            \mathsf{V}_m^{[d]} & \mathsf{v}_{m+1}^{[d]}
        \end{array}\right].
\end{equation*}
We consider a similar subdivision for $\bmZ_m\in\C^{dn\times m}$ into $\mZ_m^{[i]}\in\C^{n\times m}$ ($i=1,\dots,d$). 
\begin{definition}\label{def:rcork-Q}
    The matrix $\mQ_m\in\C^{n\times r_m}$ is an orthogonal matrix that spans the $r_m$-dimensional column space of $[ \mV_m^{[1]} \; \cdots \; \mV_m^{[d]} \; \mZ_{m-1}^{[1]} \; \cdots \; \mZ_{m-1}^{[d]} ]$.
\end{definition}
Using \cref{def:rcork-Q}, $\bmV_{m+1}$ and $\bmZ_m$ can be expressed as 
\begin{equation*}
    \bmV_{m+1} = (\Id_d \kron \mQ_{m+1})\bmU_{m+1},\quad \bmZ_m = (\Id_d \kron \mQ_{m+1})\bmS_m,
\end{equation*}
\begin{equation*}
    \text{with}\qquad \bmU_m = \left[\begin{array}{c} \mU_m^{[1]} \\ \vdots \\ \mU_m^{[d]} \end{array}\right] \in \C^{dr_m\times m},\quad \bmS_m = \left[\begin{array}{c} \mS_m^{[1]} \\ \vdots \\ \mS_m^{[d]} \end{array}\right] \in \C^{dr_{m+1}\times m}.
\end{equation*}
This leads to the introduction of the \emph{CORK quintuple}; compare with the CORK quadruple in \cite[Definition 4.2]{VanBeeumen2015}.
\begin{definition}[CORK quintuple]\label{def:cork-quintuple}
    The quintuple $(\mQ_{m+1},\bmU_{m+1},\bmS_m,\Hu_m,\Ku_m)$ with $\mQ_{m+1}\in\C^{n\times r_{m+1}}$, $\bmU_{m+1}\in\C^{dr_{m+1}\times (m+1)}$, $\bmS_m\in\C^{dr_{m+1}\times m}$ and $\Hu_m,\Ku_m\in\C^{(m+1)\times m}$ is called a CORK quintuple of order $m$ for $(\bA,\bB)$, defined by \cref{eq:linearization}, if
    \begin{enumerate}[noitemsep,nolistsep]
        \item it satisfies the CORK recurrence relation $\forall \mu\in\C$,
        \begin{equation*}
            (\bA-\mu \bB)(\Id_d\kron \mQ_{m+1})\bmS_m = (\Id_d\kron \mQ_{m+1})\bmU_{m+1}(\Ku_m-\mu\Hu_m),
        \end{equation*}
        \item $\mQ_{m+1}$ and $\bmU_{m+1}$ are orthogonal and $\mQ_{m+1}$ has full rank, and
        \item $\Ku_m$ and $\Hu_m$ are upper Hessenberg matrices with $\Hu_m$ unreduced.
    \end{enumerate}
\end{definition}
Through additional theorems we show that this representation is compact in the sense that its storage scales as $\mathcal{O}(nm+dm^2)$ instead of $\mathcal{O}(dnm)$. The compact block structure is inherited from that of the solution of the shifted system.

\begin{lemma}\label{lem:system-solve-ulp}
    Let $\bA$ and $\bB$ be defined by \cref{eq:linearization}. The solutions $\bvz$ and $\bvx$ of
    \begin{equation*}
        (\bA-\sigma\bB)\bvz = \bvy, \quad\quad \bvx = \bB\bvz.
    \end{equation*}
    have the following block structure. One block, $\vz^{[p]}$, corresponds to a system solve with $\mP(\sigma)$. The other blocks of $\bvz$ are linear combinations of $\vz^{[p]}$ and the blocks of $\bvy$. The first block of $\bvx$, $\vx^{[1]}$, is obtained by a matrix-vector product between $\bar{\bB}$ and $\bvz$, and the other blocks of $\bvx$ are linear combinations of the blocks of $\bvz$.
\end{lemma}
\begin{proof}
    The part on $\bvz$ follows from the block ULP decomposition of $\bmat{L}(\sigma) = \bA - \sigma \bB$, see the proof of Lemma 4.3 in \cite{VanBeeumen2015}. The part on $\bvx$ is clear from the structure of $\bB$.
\end{proof}
\begin{theorem}\label{thm:span-of-Q}
Let $\mQ_m$ be defined by \cref{def:rcork-Q}. Then,
\begin{equation*}
    \vspan{\{\mQ_{m+1}\}} = \vspan{\{\mQ_m,\vz_{m}^{[p]},\vv_{m+1}^{[1]}\}},
\end{equation*}
where $\vz_m^{[p]}$ is the $p$th block in \Cref{lem:system-solve-ulp}. In addition, if $\exists \vec{c}_\mB \in \C^d$ s.t.\ $\bar{\bB}=\vec{c}_\mB^\tr\kron\Id_n$,
\begin{equation*}
    \vspan{\{\mQ_{m+1}\}} = \vspan{\{\mQ_m,\vz_{m}^{[p]}\}}.
\end{equation*}
\end{theorem}
\begin{proof}
Using \cref{lem:system-solve-ulp} with $\sigma=\sigma_m$, we have
\begin{align*}
    \vspan{\{\mQ_{m+1}\}} &= \vspan{\left\{ \mV_{m+1}^{[1]},\; \dots,\; \mV_{m+1}^{[d]},\; \mZ_m^{[1]},\; \dots,\; \mZ_m^{[d]} \right\}} \\
        &= \vspan{\left\{ \mQ_m,\; \vv_{m+1}^{[1]},\; \dots,\; \vv_{m+1}^{[d]},\; \vz_m^{[1]},\; \dots,\; \vz_m^{[d]} \right\}} \\
        &= \vspan{\left\{ \mQ_m,\; \vv_{m+1}^{[1]},\; \vz_m^{[p]} \right\}}.
\end{align*}    
Clearly, when $\bar{\bB}=\vec{c}_\mB^\tr\kron\Id_n$ in \cref{lem:system-solve-ulp}, $\vx^{[1]}$ is also a linear combination of blocks of $\bvz$. Thus $\vv_{m+1}^{[1]} \in \vspan{\{\mQ_m,\vz_{m}^{[p]}\}}$.
\end{proof}
\begin{theorem}
    For $\mQ_m$ of \cref{def:rcork-Q} and $\tilde{d} = \mathrm{dim}(\vspan{\{\vv_1^{[1]},\dots,\vv_1^{[d]}\}})$,
    \begin{equation}\label{eq:inequality_rank_Qm}
        r_m \le \tilde{d} + c(m-1),
    \end{equation}
    where $c=2$ in general and $c=1$ if $\exists\vec{c}_{\mB}\in\C^d$ such that $\bar{\bB}=\vec{c}_{\mB}^\tr\kron\Id_n$.
\end{theorem}
\begin{proof}
    By definition, $\vspan{\{\mQ_1\}}=\vspan{\{\vv_1^{[1]},\dots,\vv_1^{[d]}\}}$, i.e. $r_1=\tilde{d}$. Then from \cref{thm:span-of-Q} it is clear that $r_{k+1}\le r_k + c$. By induction the statement is proven.
\end{proof}

When the shifted system $(\bA-\mu\bB)\bvy=\bvd$ is used to solve the parametrized linear system $\mP(\mu)\vx=\vb$, we have $\bvd=\vec{e}_1\kron\vb$ and $\tilde{d}=1$. In general, with $c=2$, RCORK-GMRES requires roughly twice the amount of storage to store its vectors compared to LCORK-GMRES. This is similar to how regular RRK-GMRES requires twice the storage compared to LRK-GMRES by having to store $\bmZ_m$.

In \cite{Bruyninckx2026}, the main author further details how the block ULP decomposition combined with the compact representation of $\bmV_{m+1}$ and $\bmZ_m$ leads to efficient implementations of both LCORK-GMRES and RCORK-GMRES. The computationally intensive parts are the ones that scale with the problem size $n$. They are, in iteration $k$:
\begin{enumerate}
    \item Solving a system involving $\mP(\sigma_k)$ with shift $\sigma_k$.
    \item A matrix-vector product with $(\bar{\bA}_2-\sigma_k\bar{\bB}_2)\in\C^{(d-1)n\times n}$ and one with $\bar{\bB}\in\C^{dn\times n}$. This can be recast to a single matrix-vector product with $\bar{\bmat{C}}:=[\mat{C}_1 \,\cdots\, \mat{C}_s]\in\C^{sn\times n}$ in the case of LCORK-GMRES. RCORK-GMRES requires two separate products with $\bar{\bmat{C}}$, unless $\bar{\bB}=\vec{c}_{\mB}^\tr\kron\Id_n$.
    \item Expanding $\mQ_k$ to $\mQ_{k+1}$ by means of orthogonalization. It is expanded by one vector for LCORK-GMRES and by two for RCORK-GMRES (except when $\bar{\bB}=\vec{c}_{\mB}^\tr\kron\Id_n$).
\end{enumerate}

The cost of the orthogonalization depends on the choice of orthogonalization process, but generally results in $\mathcal{O}(nm^2)$ operations over $m$ iterations. In contrast, the cost of the other operations mostly depends on the form of $\mP(\mu)$. Note that when $\bar{\bB}=\vec{c}_{\mB}^\tr\kron\Id_n$, RCORK-GMRES practically has the same cost as LCORK-GMRES. Earlier, the same was also true for their memory footprint.

\subsection{Parameter dependency of the right-hand side}\label{ssec:cork-rhs}

At first, a parameter-dependent right-hand side $\vb:=\vb(\mu)$ seems like a major obstacle. Krylov methods that rely on the shift-invariance property to solve shifted linear systems are not designed to handle this situation, since there is no single starting vector that is uniformly valid in $\mu$. To the best of our knowledge, this case is not treated in the literature on linearization.
 
Fortunately, there is an original solution to this problem, which is also relatively simple. We can make a companion-like extension of the system to a slightly larger one, with size $(n+1)\times(n+1)$, with a constant right-hand side. To that end, we assume that $\vb$ can be represented with the same basis functions as $\mP$, i.e.,
\begin{equation}\label{eq:cond-param-dep-rhs}
    \vb(\mu) = \sum_{i=1}^d (\vb_{\mA,i} - \mu \vb_{\mB,i}) f_i(\mu).
\end{equation}
We aim to construct an augmented system $\mPu(\mu)\vxu(\mu)=\vbu(\mu)$ of which the right-hand side is only dependent on $\mu$ through a scalar. We do so using a scalar function $\gamma(\mu)=\sum_{i=1}^d (\alpha_i-\mu\beta_i)f_i(\mu)$ and defining
\begin{equation*}
    \mPu(\mu):=\left[\begin{array}{cc}
        \mP(\mu) & -\vb(\mu) \\
        \mat{0}_{1\times n} & \gamma(\mu)
    \end{array}\right], \quad
    \vbu(\mu) := \gamma(\mu)\vec{e}_{n+1}=\left[\begin{array}{c}
        0 \\
        \gamma(\mu)
    \end{array}\right].
\end{equation*}
where $\vxu(\mu) = [\vx(\mu)^\tr \; 1]^\tr$.
The augmented system results in a linearization with the original matrices $\{\mA_i\}_{i=1}^d$ and $\{\mB_i\}_{i=1}^d$ also augmented to slightly larger ones. Importantly, we can perform CORK-GMRES with starting vector $\vec{e}_{n+1}$. The solution to the original system is obtained from the first $n$ elements of the augmented solution, scaled by $\gamma(\mu)$.

Any scalar function $\gamma(\mu)$ of the right form will do, but generally one can choose the coefficients $\alpha_i$ and $\beta_i$ such that $\gamma(\mu)$ is a constant. The choice could impact the stability of CORK. Strictly speaking, we have also changed the meaning of the residual of the system, which is a criterion for the greedy selection of shifts. Still, we found the approach to be very effective in the application of this paper.

\section{Inexact rational Krylov}\label{sec:4_inexact-rk}

Thus far we have assumed that the systems $\mP(\sigma)\vx=\vb$ can be solved exactly. This is often not the case, or it might be impractical, for large-size problems. Instead, iterative methods are used to obtain an approximate solution with a suitably small residual norm. This leads to inexact Krylov methods, for which there is a rich body of literature~\cite{BourasFrysse2005,ChenEtAl2016,DuSzyld2008,GrattonEtAl2020,KahlRittich2017,MansourGotze2014,SidjeWinkles2011,SimonciniSzyld2005b,SimonciniSzyld2005,VanDenEshofEtAl2004,WangEtAl2022}.

The choice of the tolerance with which to solve the systems affects accuracy and computational cost. It was shown in \cite{SimonciniSzyld2003} that one can loosen the tolerances of inexact matrix-vector products in GMRES, CG and related methods, throughout the iterations. More specifically, the required tolerances can be chosen inversely proportional to the computed residual norm of the previous iteration. This observation also applies to shift-and-invert Arnoldi for shifted linear systems~\cite{SimonciniSzyld2003}.

We adopt the strategy for RK-GMRES, similar to~\cite{SimonciniSzyld2003,GiraudEtAl2007}, and we also apply it to CORK-GMRES. We derive conditions on the residual norms in both settings.

\subsection{Inexact RK-GMRES}\label{ssec:inexact-rk}
Inexact solves in LRK-GMRES transform the exact recurrence relations \cref{eq:rat-krylov-reccurence} into a form that includes the inner residuals,
\begin{equation}\label{eq:inexact-lrk-gmres-relation}
        \bA \bmV_{m+1} \Hu_m = \bB \bmV_{m+1} \Ku_m - \bmat{P}_m.
\end{equation}
Using
$\bmat{W}_m=\bmV_{m+1}\Hu_m$, with columns $\bvec{w}_k$, $\bmat{P}_m=[\bvec{p}_1 \, \cdots \, \bvec{p}_m]$ contains the residuals
\[
\bvec{p}_k := \bB\bvv_{\vec{t},k} - (\bA-\sigma_k \bB)\bvec{w}_k, \quad \bvv_{\vec{t},k}:=\bmV_k\vec{t}_k, \quad k=1,\dots,m .
\]
The final true residual of LRK-GMRES is composed of a contribution from $\bmat{P}_m$ and the residual $\tilde{\bvec{r}}^\sigma_m$ of the least-squares problem, which can be cheaply computed,
\[
\tilde{\bvec{r}}^\sigma_m:=\bmV_{m+1}\left(\hbeta\vec{e}_1-(\Ku_m-\mu\Hu_m)\vy^\red_m\right).
\]
An appropriate choice for the inner tolerances can be established.
\begin{proposition}\label{prop:inexact-lrk-gmres}
Let $\epsilon>0$. Let $\bvec{r}_m = \bvd - (\bA-\mu \bB)\bvy_m$ be the true residual and $\tilde{\bvec{r}}^\sigma_m$ the computed residual after $m$ inexact iterations of LRK-GMRES. If the inner residual $\bvec{p}_k$ of iteration $k$ satisfies
    \begin{equation*}
        \|\bvec{p}_k\| \le \dfrac{1}{|\mu-\sigma_0|}\dfrac{\mathrm{sval}_\mathrm{min}(\Ku_m-\mu\Hu_m)}{m}\dfrac{\epsilon}{\|\tilde{\bvec{r}}^{\sigma}_{k-1}\|}, \qquad k=1,\ldots,m,
    \end{equation*}
    then
    \[
    \|\bvec{r}_m\|\le\|(\bA-\sigma_0\bB)\tilde{\bvec{r}}^\sigma_m\|+\epsilon.
    \]
\end{proposition}
\begin{proof}
We transform \cref{eq:inexact-lrk-gmres-relation} as follows
\begin{align*}
    &\bA \bmV_{m+1} \Hu_m - \sigma_0\bB \bmV_{m+1} \Hu_m = \bB \bmV_{m+1} \Ku_m - \sigma_0\bB \bmV_{m+1} \Hu_m - \bmat{P}_m \\
    &\Leftrightarrow (\bA- \sigma_0\bB) \bmV_{m+1} \Hu_m = \bB \bmV_{m+1} (\Ku_m - \sigma_0\Hu_m) - \bmat{P}_m \\
    &\Leftrightarrow (\sigma_0-\mu)(\bA- \sigma_0\bB) \bmV_{m+1} \Hu_m +(\bA-\sigma_0 \bB)\bmV_{m+1}(\Ku_m-\sigma_0\Hu_m) \\
    &\quad = (\sigma_0-\mu)\bB \bmV_{m+1} (\Ku_m - \sigma_0\Hu_m) + (\bA-\sigma_0 \bB)\bmV_{m+1}(\Ku_m-\sigma_0\Hu_m) - (\sigma_0-\mu)\bmat{P}_m \\
    &\Leftrightarrow (\bA-\sigma_0 \bB)\bmV_{m+1}(\Ku_m-\mu\Hu_m) = (\bA-\mu \bB)\bmV_{m+1}(\Ku_m-\sigma_0 \Hu_m) + (\mu-\sigma_0) \bmat{P}_m.
\end{align*}
Using $\bvy_m=\bmV_{m+1}(\Ku_m-\sigma_0\Hu_m)\vy^\red_m$, the final residual vector $\bvec{r}_m$ becomes
\begin{align*}
    \bvec{r}_m &= \bvd - (\bA-\mu \bB)\bvy_m \\
        &= \bvd -(\bA-\mu \bB)\bmV_{m+1}(\Ku_m-\sigma_0\Hu_m)\vy^\red_m \\
        &= \bvd -(\bA-\sigma_0 \bB)\bmV_{m+1}(\Ku_m-\mu\Hu_m)\vy^\red_m + (\mu-\sigma_0)\bmat{P}_m\vy^\red_m \\
        &= (\bA-\sigma_0 \bB)\tilde{\bvec{r}}^\sigma_m + (\mu-\sigma_0)\bmat{P}_m\vy^\red_m \\
        &= (\bA-\sigma_0 \bB)\tilde{\bvec{r}}^\sigma_m + (\mu-\sigma_0) \sum_{k=1}^m  y^\red_{m,k} \bvec{p}_k.
\end{align*}
Similarly to \cite[\S2.1]{GiraudEtAl2007}, the magnitudes of the elements of $\vy_m^\red$ can be bounded as $|y_{m,k}^\red|\le\|(\Ku_m-\mu\Hu_m)^\dagger\|\|\tilde{\bvec{r}}_{k-1}\|$, $k=1,\dots,m$. The result follows.
\end{proof}
As was to be expected, there is a gap between the computed residual $\tilde{\bvec{r}}^\sigma_m$ and the true final residual $\bvec{r}_m$. 
\Cref{prop:inexact-lrk-gmres} does not take into account solving the first system with $\sigma_0$ inexactly. It is clear that an inner residual $\bvec{p}_0$ introduced when solving the system with $(\bA-\sigma_0\bB)$, propagates into a contribution to $\|\bvec{r}_m\|$ of at most $\|\bvec{p}_0\|$.

The case of RRK-GMRES is similar, though somewhat simpler as the residual is not preconditioned. With inexact solves the recurrence relation \cref{eq:rrkgmres-shifted-recurrence} becomes
\begin{equation*}
    (\bA-\mu \bB)\bmZ_m = \bmV_{m+1}(\Ku_m-\mu \Hu_m) - \bmat{P}_m, \qquad \bvec{p}_k:=\bvv_{\vec{t},k}-(\bA-\sigma_k \bB)\bvz_k.
\end{equation*}
For a solution $\bvy_m(\mu)=\bmZ_m\vy^\red_m(\mu)$, the final true residual $\bvec{r}_m:=\bvec{r}_m(\mu)$ becomes
\begin{equation*}
    \bvec{r}_m = \bvd - (\bA-\mu \bB)\bmZ_m\vy^\red_m = \bmV_{m+1}(\beta \vec{e}_1 - (\Ku_m-\mu\Hu_m)\vec{y}^\red_m) + \bmat{P}_m\vy^\red_m = \tilde{\bvec{r}}_m + \bmat{P}_m\vy^\red_m
\end{equation*}
where $\tilde{\bvec{r}}_m:=\tilde{\bvec{r}}_m(\mu)$ denotes the computed residual of the small least-squares problem.

The following is similar to the result derived in \cite{Correnty2024}.
\begin{proposition}
    Let $\epsilon>0$. Let $\bvec{r}_m$ be the true residual and $\tilde{\bvec{r}}_m$ the computed residual after $m$ inexact iterations of RRK-GMRES. If the inner residual $\bvec{p}_k$ satisfies
    \begin{equation*}
        \|\bvec{p}_k\| \le \dfrac{\mathrm{sval}_\mathrm{min}(\Ku_m-\mu\Hu_m)}{m}\dfrac{ \epsilon}{\|\tilde{\bvec{r}}_{k-1}\|}
    \end{equation*}
    for $k=1,\dots,m$, then $\|\bvec{r}_m\|\le\|\tilde{\bvec{r}}_m\|+\epsilon$.
\end{proposition}
\begin{proof}
    A bound on the true final residual norm is given by
\begin{equation*}
    \|\bvec{r}_m\| \le \|\tilde{\bvec{r}}_m\| + \|\bmat{P}_m\vy^\red_m\| \le \|\tilde{\bvec{r}}_m\| + \sum_{k=1}^m  |y_{m,k}^\red| \|\bvec{p}_k\|.
\end{equation*}
The result again follows by bounding $|y^\mathrm{red}_{m,k}|$, $k=1,\dots,m$, as in \cite[\S2.1]{GiraudEtAl2007}.
\end{proof}
\begin{remark}
    Implicitly, the residuals $\tilde{\bvec{r}}_{k-1}$, $\tilde{\bvec{r}}^\sigma_{k-1}$ and $\bvec{r}_m$ are all $\mu$-dependent. To guarantee that $\|\bvec{r}_m(\mu)\|\le\|\tilde{\bvec{r}}_m(\mu)\|+\epsilon$ holds for all $\mu\in S$, the bound on $\|\bvec{p}_k\|$, $k=1,\dots,m$, must be minimized over $S$. The factor $\mathrm{sval}_\mathrm{min}(\Ku_m-\mu\Hu_m)$ is not known a priori. Following \cite{SimonciniSzyld2003} and references therein, it is replaced by a constant in practice.
\end{remark}
\begin{remark}
    One may require a relative tolerance on the final true residual norm by scaling the bound on $\|\bvec{p}_k\|$ by $\|\bvd\|$. The bound can then be considered a relative one, since $\|\bvv_{\vec{t},k}\|=1$. Both observations do not hold for LRK-GMRES. Here, the true residual is not a simple sum of two contributions and $\|\bB\bvv_{\vec{t},k}\|\ne1$.
\end{remark}

\subsection{Inexact CORK-GMRES}\label{ssec:inexact-cork}
When $(\bA,\bB)$ is the companion linearization of a matrix-valued function $\mP$, the residual $\bvec{p}=\bvy-(\bA-\mu \bB)\bvz$, $\|\bvec{p}\|\le\tau$ originates from an inexact solve with $\mP(\mu)$. Under the assumption that all other operations in efficiently applying $(\bA-\mu \bB)^{\minone}$ are performed exactly, the following holds.
\begin{proposition}\label{prop:inexact-cork}
    Let $\tau>0$. Let $(\bA,\bB)$ be the companion linearization defined by \cref{eq:linearization} for matrix-valued function $\mP$ \cref{eq:matrix-function-add}. Assume we solve a system with matrix $\bA-\sigma\bB$ and right-hand side $\bvy$, using the block ULP decomposition (\cref{thm:ULP}) where all but the system solution with $\mP(\sigma)$ is performed exactly. Denote the approximate solution as $\bvz$ with residual $\bvec{p}=\bvy-(\bA-\mu \bB)\bvz$. Then,
    \begin{equation*}
        \|\vec{q}\| \le |\alpha|\tau, \quad \vec{q}=\vec{w}^{[1]}-\mP(\mu)\vz^{[p]}, \quad \vec{w}^{[1]} = \alpha (\vy^{[1]}-\cU_{12}\vy^{[2:d]})
    \end{equation*}
    guarantees $\|\bvec{p}\|\le\tau$, where $p$ is s.t.\ $\cP\vec{e}_p=\vec{e}_1$, $\alpha=\vec{e}_1^\tr\cP\vf(\sigma)$ and $\cU_{12}$ is the upper-right block of $\cU$.
\end{proposition}
\begin{proof}
    For ease of exposition we disregard the permutation $(\cP\kron I_n)$ s.t.\ $p=1$ and  write
    \begin{equation*}
        \cL =: \left[\begin{array}{cc}
            \mP(\mu) & \mat{0} \\
            \cL_{21} & \cL_{22}
        \end{array}\right], \quad\quad\quad 
        \cU =: \left[\begin{array}{cc}
            \alpha^{\minone}\Id_n & \cU_{12} \\
            \mat{0} & \Id_{(d-1)n}
        \end{array}\right].
    \end{equation*}
$\bvz\approx(\cU\cL)^{\minone}\bvy=\cL^{\minone}\cU^{\minone}\bvy$ is computed as $\bvec{w}=\cU^{\minone}\bvy$ and $\bvz\approx\cL^{\minone}\bvec{w}$. We get that 
$$\vec{w}^{[2:d]}=\vy^{[2:d]} \quad\text{and}\quad \vec{w}^{[1]} = \alpha (\vy^{[1]}-\cU_{12}\vy^{[2:d]}).$$ 
Next, $\vz^{[1]}\approx \mP(\mu)^{\minone}\vec{w}^{[1]}$ such that $\vec{q}=\vec{w}^{[1]}-\mP(\mu)\vz^{[1]}$, i.e.\ 
$$\vz^{[1]}=\mP(\mu)^{\minone}(\vec{q}-\vec{w}^{[1]}). \quad\text{Also,}\quad \vz^{[2:d]} = \cL_{22}^{\minone}(\vec{w}^{[2:d]}-\cL_{21}\vz^{[1]}).$$
Now computing the residual $\bvec{p}=\bvy-(\cU\cL)\bvz$, 
$$\cL\bvz = \left[\begin{array}{c}
    \vec{q}-\vec{w}^{[1]} \\
    -\vec{w}^{[2:d]}
\end{array}\right] = \left[\begin{array}{c}
    \vec{q}-\alpha (\vy^{[1]}-\cU_{12}\vy^{[2:d]}) \\
    -\vec{y}^{[2:d]}
\end{array}\right]
$$ 
and thus $\bvec{p} = [\alpha^{\minone}\vec{q}^\tr \; \vec{0}]^\tr$.
\end{proof}

In the CORK setting, emphasis must be placed on the fact that $\bvec{r}_m$ is not actually the residual of our original problem. Instead, it is residual $\vec{r}_m^\mP:=\vb-\mP(\mu)\vx_m$ that is of importance. Due to the expected form of $\bvy_m$, see \cref{thm:linearization}, one generally takes $\vx_m:=\vy_m^{[1]}/f_1(\mu)$. If the approximate solution is exactly representable as $\bvy_m = \vf(\mu)\otimes \vx_m$, then $\|\vec{r}^\mP_m\|=\|\bvec{r}_m\|$ ($\vec{r}_m^{[1]} = \vec{r}_m^\mP,\; \vec{r}_m^{[2:d]}=\vec{0}$), thus we expect $\|\vec{r}^\mP_m\|\approx\|\bvec{r}_m\|$.

When the right-hand side is parameter-dependent as discussed in \cref{ssec:cork-rhs}, solving the augmented system can clearly be achieved through solving a system with $\mP(\mu)$ with some additional cheap operations. However, when considering the inexact solves, the previous conditions relate to the augmented system, not the original.

The final true residual on the augmented system can be written as follows, assuming that CORK-GMRES is applied with starting vector $\vec{e}_{n+1}$,
\begin{align*}
    \vru_m^\mP(\mu) := \dfrac{\vbu(\mu)-\mPu(\mu)\vxu_m(\mu)}{\gamma(\mu)} &= \left[\begin{array}{c}
        \mat{0} \\
        1
    \end{array}\right] - 
    \left[\begin{array}{cc}
        \mP(\mu) & -\vb(\mu) \\
        \mat{0}_{1\times n} & \gamma(\mu)
    \end{array}\right]
    \left[\begin{array}{c}
       \vec{x}_m/\gamma(\mu) \\
        1/\gamma(\mu)
    \end{array}\right] \\
    &= \left[\begin{array}{c}
        (\vb(\mu) - \mP(\mu)\vx)/\gamma(\mu) \\
        0
    \end{array}\right] =
    \left[\begin{array}{c}
        \vr_m^\mP(\mu)/\gamma(\mu) \\
        0
    \end{array}\right].
\end{align*}
This shows that the true residual norm on the original system is scaled by $1/\gamma(\mu)$ and thus the bounds on the inner residual norms should scale by $1/|\gamma(\mu)|$. As the final element of the approximate solution is not exactly equal to $1/\gamma(\mu)$ in general, we only expect an approximate equality of $\|\vru_m^\mP(\mu)\|\approx\|\vr_m^\mP(\mu)\|/|\gamma(\mu)|$.

Besides relating the final true residual $\vru_m^\mP$ of the augmented system to that of the original system, we must do the same for the conditions on the inner residuals of the inexact solves. Consider the residual $\vqu = \underline{\vec{w}}^{[1]} - \mPu(\mu)\underline{\vz}^{[1]}$ on the augmented system\footnote{Note that we use similar notation as in \cref{prop:inexact-cork}.},
\begin{equation*}
    \vqu = \left[\begin{array}{c}
        \vec{w} \\
        w_{n+1}
    \end{array}\right] -
    \left[\begin{array}{cc}
        \mP(\mu) & -\vb(\mu) \\
        \mat{0}_{1\times n} & \gamma(\mu)
    \end{array}\right]
    \left[\begin{array}{c}
        \vz \\
        z_{n+1}
    \end{array}\right] = 
    \left[\begin{array}{c}
        \vec{w}+z_{n+1}\vb(\mu)-\mP(\mu)\vz \\
        w_{n+1}-\gamma(\mu) z_{n+1}
    \end{array}\right].
\end{equation*}
It is clear that we can choose $z_{n+1}=w_{n+1}/\gamma(\mu)$ such that $\|\vqu\|=\|\vq\|$ where $\vq$ is the residual on the inexact system solve of $\mP(\mu)\vz\approx \vec{w}+\frac{w_{n+1}}{\gamma(\mu)}\vb(\mu)$. Note that the residual norms are equivalent but the right-hand sides are not, such that the required relative tolerance is different from the non-augmented case.

We conclude by summarizing the practical bounds. In LCORK-GMRES applied to $\mP$ with right-hand side $\vb$, the condition
\begin{equation*}
    \|\vp_k\| \le \min_{\mu\in S}\dfrac{|\trns{\vec{e}}_1\cP\vf(\sigma_k)|}{|\mu-\sigma_0|} \dfrac{\text{sval}_\text{min}(\Ku_m-\mu\Hu_m)}{m} \dfrac{\|\vb\|}{\|\tilde{\bvec{r}}^{\sigma}_{k-1}(\mu)\|} \epsilon,\quad k=1,\dots,m,
\end{equation*}
where $\vp_k$ is the inner residual of the inexact solve with $\mP(\sigma_k)$ at iteration $k$, guarantees\footnote{Remember that these ``guarantees" hinge on the assumption that the full approximate solution $\bvy_m$ is representable as $\vf(\mu)\otimes\vx_m$ for some $\vx_m$. These assumptions are practically never satisfied exactly, only approximately.} $\|\vr_m^\mP(\mu)\|\le \|(\bA-\sigma_0 \bB)\tilde{\bvec{r}}^{\sigma}_m(\mu)\| + \epsilon \|\vb\|$ for all $\mu\in S$. Similarly for RCORK-GMRES,
\begin{equation*}
    \|\vp_k\| \le \min_{\mu\in S}|\trns{\vec{e}}_1\cP\vf(\sigma_k)| \dfrac{\text{sval}_\text{min}(\Ku_m-\mu\Hu_m)}{m} \dfrac{\|\vb\|}{\|\tilde{\bvec{r}}_{k-1}(\mu)\|} \epsilon,\quad k=1,\dots,m,
\end{equation*}
guarantees $\|\vr_m^\mP(\mu)\|\le\|\tilde{\bvr}_m(\mu)\|+\epsilon \|\vb\|$ for all $\mu\in S$. For parameter-dependent right-hand sides, when the methods are applied to $\mPu$ with right-hand side $\vec{e}_{n+1}$, the bounds should be divided by $|\gamma(\mu)|$ and $\|\vb\|=\|\vb(\mu)\|$ is also $\mu$-dependent.

\begin{remark}\label{rmk:inner-rel-tol}
    In the numerical experiments, we always report the inner \emph{relative} tolerance on $\|\vp_k\|$ because this better indicates the cost of the system solution. This corresponds to the above bounds, divided by the norm of the system's right-hand side. These relative tolerances do not necessarily increase monotonically.
\end{remark}

\subsection{Numerical experiment: parametrized Helmholtz}\label{ssec:cork-numexp}

We illustrate the effectiveness of LCORK- and RCORK-GMRES, and the incorporation of inexact Krylov theory, by comparing to the inexact infinite GMRES method from \cite{Correnty2024}. We repeat the experiment from Figure 6.7(b) of \cite{Correnty2024}, which originates from a finite element discretization of a 2D parametrized Helmholtz equation (see \S6.2 for details\footnote{Implementation note: to obtain the large system, we have set `200' to `400' on line 103 of \texttt{fem\_matrices.py} at \href{https://github.com/siobhanie/FlexInfGMRES}{github.com/siobhanie/FlexInfGMRES}.}).

We perform CORK-GMRES for $\mu\in[0,2.5]$ with shifts $\bm{\sigma}=(0.25,1.25,2.25)$. We use both the regular versions with LU factorizations and the inexact versions with BiCGStab. In the latter, bounds in \S\ref{ssec:inexact-cork} are used for the inner tolerances, with $\epsilon=10^{-9}$ and the unknown smallest singular values replaced by one. 
In LCORK-GMRES, $\sigma_0$ is set to $1.25$ and this initial inexact solve is performed using a relative tolerance of $10^{-10}$. The sine function in the formulation is approximated with AAA~\cite{NakatsukasaEtAl2018}, using a tolerance of $10^{-13}$, which yields a rational approximant with $8$ support points. The resulting linearization differs significantly from the Taylor series linearization in \cite{Correnty2024}.

The results are shown in \cref{fig:numexp-phfem}. We consider the maximum relative residual norm $\|\vec{r}_k^\mP(\mu)\|/\|\vb\|=\|\vb-\mP(\mu)\vx_k(\mu)\|/\|\vb\|$ over $\mu\in[0,2.5]$. Comparing to (inexact) infinite GMRES, the number of required iterations to achieve similar residual norms is significantly decreased, owing to the absence of initial stagnation. In the exact case, a relative residual norm of $10^{-10}$ is achieved in $20$ iterations rather than $40$. The inner relative tolerances are shown to increase throughout the iterations. The iteration at which the inexact version stagnates  between $10^{-9}$ and $10^{-10}$  aligns well with the inner tolerance reaching a value around 1.

\begin{figure}[tbp]
\begin{tikzpicture}
\begin{groupplot}[group style={group size=2 by 1,
                    ylabels at=edge left,
                    yticklabels at=edge left,
                    horizontal sep=0.22\linewidth},
                  xlabel = {Iteration $k$},
                  ylabel = {\small$\max_\mu{\|\vec{r}^\mP_k(\mu)\| \, /\, \|\vb\|}$},
                  legend style={legend columns=1}]

\nextgroupplot[width=0.32\linewidth,
              ymin=1e-13, ymax=1e1,
              ytick={1e-1,1e-4,1e-7,1e-10,1e-13},
              scale only axis,
              ymode=log,
              grid=both,
              grid style = {dotted,gray},
              legend to name=grouplegend,
              legend style = {font=\small}]

    \pgfplotstableread[col sep=comma]{data/fig1-phfem/phfem-icork_left-1.0e-9_0.001.csv}{\Ldata} 

        \addplot [color=T-Q-HC3,mark=o,mark options=solid,thick] table 
            [skip first n=1, x index=0, y expr=\thisrowno{2}] 
            {\Ldata};\addlegendentry{BiCGStab};
        \addplot [color=T-Q-HC4,dashed,mark=x,mark options=solid,thick] table 
            [skip first n=1, x index=0, y expr=\thisrowno{1}] 
            {\Ldata};\addlegendentry{LU};
        \addplot [color=T-Q-HC0,dotted,thick] table 
            [skip first n=1, x index=0, y expr=\thisrowno{3}] 
            {\Ldata};\addlegendentry{inner tol};

\nextgroupplot[width=0.32\linewidth,
               ymin=1e-13, ymax=1e1,
               ytick={1e-1,1e-4,1e-7,1e-10,1e-13},
               scale only axis,
               ymode=log,
               grid=both,
               grid style = {dotted,gray}]
                    
    \pgfplotstableread[col sep=comma]{data/fig1-phfem/phfem-icork_right-1.0e-9_0.001.csv}{\Rdata}

        \addplot [color=T-Q-HC3,mark=o,mark options=solid,thick] table 
            [skip first n=1, x index=0, y expr=\thisrowno{2}] 
            {\Rdata};
        \addplot [color=T-Q-HC4,dashed,mark=x,mark options=solid,thick] table 
            [skip first n=1, x index=0, y expr=\thisrowno{1}] 
            {\Rdata};
        \addplot [color=T-Q-HC0,dotted,thick] table 
            [skip first n=1, x index=0, y expr=\thisrowno{3}] 
            {\Rdata};
\end{groupplot}
\node [anchor=south] at ($(group c1r1.east)!0.5!(group c2r1.west)$) {\pgfplotslegendfromname{grouplegend}};
\end{tikzpicture}
\caption{Maximum relative residual norms of LCORK-GMRES (left) and RCORK-GMRES (right) applied to the parametrized Helmholtz problem, where system solves are either done exactly using an LU factorization or inexactly through BiCGStab. The inner relative tolerances (Remark \ref{rmk:inner-rel-tol}) for BiCGStab are also shown.}\label{fig:numexp-phfem}
\end{figure}
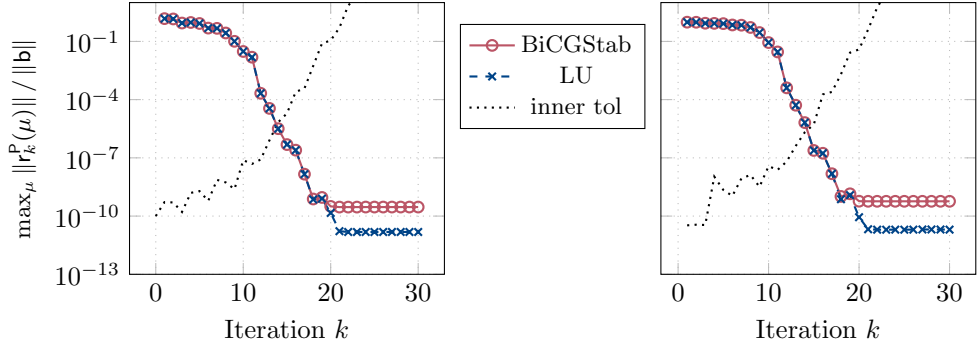

\section{Application to Helmholtz BEM}\label{sec:5_bem}

As main application of this manuscript, we consider the solution of scattering problems defined by the 3D Helmholtz equation, through the use of the Boundary Element Method (BEM).

\subsection{Problem definition}\label{ssec:bem-problem}

The application under consideration is that of acoustic (exterior) scattering problems in 3D, solved efficiently using their integral equation formulation. The Helmholtz equation $\Delta u + \kappa^2 u = 0$ is considered in an open domain $\Omega\subseteq\R^3$ where total field $u$ is composed of a given incident field $u^I$ and an unknown scattered field $u^S$, i.e.\ $u:=u^I+u^S$. Boundary conditions are imposed on $u$ and $u^S$ must satisfy the Sommerfeld radiation condition $\lim_{\mathbf{r}\to\infty}{\mathbf{r}\cdot\left(\partial u^S/\partial\mathbf{r} - i\kappa u^S\right)}=0$.
The elliptic boundary value problem can be reformulated as a boundary integral equation (BIE) \cite[\S7]{McLean2000} with an integral operator of the form
\begin{equation*}
    (\mathcal{A}_\kappa\phi)(\px) = \int_{\Gamma}g_\kappa(\px,\py)\phi(\py)\,\mathrm{d}s_\py, \;\; \forall \px\in\Gamma:=\partial\Omega
\end{equation*}
where the kernel function $g_\kappa$ involves the 3D Helmholtz Green's kernel
$\Gk(\px,\py) = {\mathrm{exp}({i\kappa\|\px-\py\|})}/{4\pi\|\px-\py\|}$
and/or its derivative(s).

To solve the BIE, it is discretized. Galerkin BEM employs two function bases $\{\psi_i\}_{i=1}^n$ and $\{\phi_j\}_{j=1}^n$ and results in a dense matrix $\mat{A}:=\mat{A}_\kappa$ with elements
\begin{equation*}
    a_{ij} = \left\langle \mathcal{A}_\kappa\phi_j,\psi_i \right\rangle = \int_\Gamma{\int_\Gamma{g_\kappa(\px,\py)\psi_i(\px)\phi_j(\py)\mathrm{d}s_\px\mathrm{d}s_\py}}, \quad i,j=1,\dots, n.
\end{equation*}

\subsection{System matrix representation}\label{ssec:bem-system-matrix}

While dense, the matrix $\mat{A}$ arising from Galerkin discretization exhibits data sparsity in the fact that it contains many large subblocks of approximate low rank. The $\mathcal{H}$-matrix format, introduced in \cite{Hackbusch1999}, exploits this by representing these blocks explicitly by an approximate low-rank factorization. The $\mathcal{H}$-matrix approximation $\mat{A}^{\mathcal{H}}$ of $\mat{A}$ is characterized as
\begin{equation*}
    \restr{\mat{A}^{\mathcal{H}}}{b} := \restr{\mat{A}}{b} \;\;\forall b\in\mathcal{L}^-, \quad\quad \restr{\mat{A}^{\mathcal{H}}}{b} := \mat{X}_b\mat{Y}_b^\cj \approx \restr{\mat{A}}{b} \;\;\forall b\in\mathcal{L}^+,
\end{equation*}
where $b:=\bt\times \bs,\; \bt,\bs\subset \{1,\dots,n\}$, $\mat{X}_b\in\C^{|\bt|\times k_b},\mat{Y}_b\in\C^{|\bs|\times k_b},\; k_b\le k=\mathcal{O}(1)$, and the set $\mathcal{L}:=\mathcal{L}^+\cup \mathcal{L}^-$ tiles the whole matrix. This representation allows an approximate representation of $\mat{A}$ in $\mathcal{O}(n\log n)$ complexity instead of $\mathcal{O}(n^2)$~\cite{BormGrasedyck2005}. The same complexity is achieved for the matrix-vector product.

In \cite{Dirckx2022}, a compact representation is presented for matrix-valued functions such as $\mat{A}(\kappa):=\mat{A}_\kappa$, based on a `$\mathcal{H}$-tensor' format and a `frequency extraction' technique to analytically remove most of the wavenumber dependence. This technique is not compatible with linearization, hence we employ the representation without. The $\mathcal{H}$-matrix format of the BEM matrix for a given frequency is extended into a $\mathcal{H}$-tensor format, in which the continuous wavenumber dependence is approximated by low-degree rational functions. We keep the discussion on $\mathcal{H}$-tensors here brief, full technical details are provided in~\cite{Bruyninckx2026}.

The $\mathcal{H}$-tensor approximation $\mat{A}^\mathcal{H}(\kappa)$ of a matrix-valued function $\mat{A}(\kappa)$ is defined using the same subdivision into subblocks as with $\mathcal{H}$-matrices, such that 
\begin{equation}\label{eq:htensor-original}
    \restr{\mat{A}^{\mathcal{H}}}{b}(\kappa) := \mat{U}_b\left(\sum_{\nu=1}^{R_b} \mat{G}_{b,\nu} \cdot \vec{h}_{b,\nu}^\tr \vec{g}(\kappa) \right)\mat{V}_b^\cj, \quad \forall b\in\mathcal{L}^+
\end{equation}
with $\mat{U}_b\in\C^{|\bt|\times \ell_b},\mat{V}_b\in\C^{|\bs|\times r_b}$ orthogonal, $\{\mat{G}_{b,\nu}
\}_{\nu=1}^{R_b}$ possibly non-square matrices, and $\vec{h}_{b,\nu}=[h_{b,\nu}(z_1) \,\cdots\, h_{b,\nu}(z_d)]^\tr$ and $\vec{g}(\mu)=\trns{[g_1(\mu) \,\cdots\, g_d(\mu)]}$, where
\begin{equation}\label{eq:bary-basis-funs}
g_j: \C \to \C : \mu \mapsto \left. \dfrac{w_j}{\mu-z_j} \middle/ \sum_{i=1}^d \dfrac{w_i}{\mu-z_i} \right. \quad \forall j\in\{1,\dots,d\}
\end{equation}
are the barycentric rational basis functions with support points $\{z_j\}_{j=1}^d$ and weights $\{w_j\}_{j=1}^d$. For $b\in\mathcal{L}^-$, we have $\restr{\mat{A}^{\mathcal{H}}}{b}(\kappa):=\sum_{\nu=1}^{R_b} \mat{A}_{b,\nu} \vec{h}_{b,\nu}^\tr \vec{g}(\kappa)$.

Similarly to how blocks of $\mathcal{H}$-matrices can be constructed algebraicly in $\mathcal{O}(n\log n)$ time through the use of Adaptive Cross Approximation (ACA)\cite{BebendorfRjasanow2003}, these $\mathcal{H}$-tensor blocks can be efficiently constructed using tensorial ACA variants. After performing QR-AAA~\cite{Dirckx2024b}, an accelerated version of set-valued AAA~\cite{LietaertEtAl2022} for simultaneous rational approximation of multiple scalar functions, we obtain the representation as given above, where all blocks share the $d$ basis functions $\vec{g}$.

\subsection{Incorporating the excitation}\label{ssec:bem-excitation} 

The right-hand side of a scattering BIE stems from the incident field $u^I$. However, generally, $u^I$ depends on the wavenumber $\kappa$, i.e.\ $u^I(\px):=u^I(\px;\kappa)$, just like the integral operator. When Galerkin BEM is employed, this results in a wavenumber-dependent right-hand side vector $\vb(\kappa)$.

To employ the augmented system technique (\cref{ssec:cork-rhs}), the right-hand side $\vb(\kappa)$ must satisfy condition \cref{eq:cond-param-dep-rhs}, or at least, an approximation to $\vb$ of this form must be constructed. This can be ensured by incorporating it into the input to QR-AAA. The end result is an approximation involving the same basis functions \cref{eq:bary-basis-funs}, such that
\begin{equation}\label{eq:rhs-affine}
    \vb(\kappa) \approx \tilde{\vb}(\kappa) := \mB^\tr \vec{g}(\kappa),\quad \mB \in \C^{d\times n}, \quad b_{ij}=b_j(z_i).
\end{equation}

The addition of the right-hand side into the QR-AAA approximation may increase the degree of the final rational approximant, which impacts the size of the linearization and thus the cost of CORK. We will not examine the degree increase in this manuscript as we are specifically interested in the CORK-based frequency sweep, which requires including the right-hand side anyway.

\subsection{Linearization and CORK application}\label{ssec:bem-cork}

With wavenumber-dependent representations of the system matrix \cref{eq:htensor-original} and right-hand side vector \cref{eq:rhs-affine}, a linearization can be constructed and CORK-GMRES performed. The parametrized system matrix of the augmented system is defined as
\begin{equation*}
    \mPu(\kappa) = \left[\begin{array}{cc}
        \mA^{\Hm}(\kappa) & -\tilde{\vb}(\kappa)  \\
         \vec{0}_{1\times n} & \gamma(\kappa) 
    \end{array}\right] = \sum_{i=1}^d \mCu_i g_i(\kappa), \quad \mCu_i = \left[\begin{array}{cc}
        \mC_i & -\vb_i  \\
         \vec{0}_{1\times n} & \gamma_i 
    \end{array}\right]
\end{equation*}
and for all $i=1,\dots,d$, $\vb_i = \vb(z_i)$, 
\begin{equation*}
    \restr{\mC_i}{b} = \mU_b \left(\sum_{\nu=1}^{R_b}h_{b,\nu}(z_i) \mat{G}_{b,\nu}\right)\conj{\mV}_b \;\; \forall b\in\mathcal{L}^+, \quad \restr{\mC_i}{b} = \sum_{\nu=1}^{R_b} h_{b,\nu}(z_i) \mat{A}_{b,\nu} \;\; \forall b\in\mathcal{L}^-.
\end{equation*}
To avoid scaling issues, the functions $\vf$ in $\cref{eq:matrix-function-add}$ are defined as $f_1 := 1$ and $f_i := g_{i-1}$, $i=2,\dots,d$. One can map between $\vec{g}$ and $\vf$ because $\sum_{i=1}^d g_i(\kappa) \equiv 1$. The coefficient matrices in $\cref{eq:matrix-function-add}$ are easily derived. We choose $\gamma(\kappa) \equiv 1$ by defining $\gamma_i:=1$ for all $i=1,\dots,d$. The particular matrix pencil $(\mM,\mN$) such that $(\mM-\kappa\mN)\vf(\kappa)$, is taken from \cite{LietaertEtAl2022}, although it is slightly modified: $\vec{g}$ is defined with weights in the numerator, resulting in a rescaling of rows and columns in $(\mM,\mN)$; $g_d$ is not present in $\vf$, removing a row and column of $(\mM,\mN)$. This is the final ingredient of the linearization $(\bA,\bB)$.

\begin{remark}\label{rmk:scaling}
    There are multiple ways of rescaling the linearization. Firstly $\gamma(\kappa)\equiv\bar{\gamma}$ could be chosen with $\bar{\gamma}\ne1$. Additionally, the pencil $(\mM,\mN)$ can have its rows scaled and combined arbitrarily. Finally, the right-hand side $\tilde{\vb}(\kappa)$ could also be rescaled by a constant. All these modifications may affect convergence behavior. We have not completely analyzed their effect and chose scaling factors that were observed to work well in the experiments of this paper. Convergence of LCORK-GMRES is less affected by scaling due to the preconditioning with $(\bA-\sigma_0\bB)^{\minone}$, but the inexact Krylov technique is affected in both versions. See \cite{GrammontEtAl2011} for some initial observations on scaling linearizations for parametrized systems.
\end{remark}

CORK relies on the capability to solve a system with $\mA^{\Hm}$ and perform a matrix-vector product with $\bbC$. The system solve is performed inexactly through an iterative solver. For each shift we obtain a cheap matrix-vector product by `slicing' the tensor: we construct a $\mathcal{H}$-matrix from the $\mathcal{H}$-tensor at the shift (see \cite{Dirckx2022,Dirckx2024,Bruyninckx2026}). This is inexpensive compared to constructing the $\mathcal{H}$-matrix from scratch. 

With the $\mathcal{H}$-matrix in hand, flexible GMRES (FGMRES) \cite{Saad1993} is employed on the $\mathcal{H}$-tensor, where the flexible preconditioner is BiCGStab(2) \cite{GutKnecht1993,Sleijpen1993} applied to the $\mathcal{H}$-matrix. This is necessary as the recompression performed to obtain the $\mathcal{H}$-matrix results in a numerical difference between it and the $\mathcal{H}$-tensor.\footnote{Solving a system with the $\mathcal{H}$-matrix directly is possible as well, as one can expand the inexact Krylov theory to also take into the account the error on the system matrix.} Using FGMRES allows us to use a fixed tolerance for both the $\mathcal{H}$-matrix compression as well as the inner BiCGStab(2) solution, throughout all CORK iterations. This lends itself nicely to the comparison with naively solving the systems at each wavenumber separately. In the numerical experiments, we show that FGMRES needs very few iterations to converge.

Performing the matrix-vector product $\vy=\bbC\bvx$ for $\vy\in\C^n$ and $\bvx\in\C^{dn}$ is the second ingredient. We refer to this as \emph{coefficient matrix application} (CMA). The operation can be defined in terms of matrix-vector products on the level of the blocks. The operation to perform for each block is
\begin{equation*}
    \restr{\vy}{\bt}^{(b)} = \sum_{i=1}^d \restr{\mC_i}{b} \restr{\vx}{\bs}^{[i]}, \quad \forall b = \bt\times\bs \in \mathcal{L}
\end{equation*}
after which all the computed subvectors are appropriately summed, similar to a normal $\Hm$-matrix-vector product. For admissible blocks, one can reorder the operation into \begin{equation*}
    \restr{\vy}{\bt}^{(b)} = \sum_{i=1}^d \mU_b \left(\sum_{\nu=1}^{R_b}h_{b,\nu}(z_i) \mat{G}_{b,\nu}\right)\conj{\mV}_b \restr{\vx}{\bs}^{[i]} = \mU_b \left[\begin{array}{ccc}\mat{G}_{b,1}&\cdots&\mat{G}_{b,R_b}\end{array}\right]\mathrm{vec}\left(\conj{\mV}_b\restr{\mat{X}}{\bs}\mat{H}_b\right),
\end{equation*}
\begin{equation*}
    \restr{\mat{X}}{\bs} = \left[\begin{array}{ccc} 
        \restr{\vx}{\bs}^{[1]} & \cdots & \restr{\vx}{\bs}^{[d]} 
    \end{array}\right], \quad 
    \mat{H}_b :=\left[\begin{array}{ccc}
        h_{b,1}(z_1) & \cdots & h_{b,R_b}(z_1)  \\
        \vdots & \ddots & \vdots \\
        h_{b,1}(z_d) & \cdots & h_{b,R_b}(z_d) 
    \end{array}\right] \quad\forall b\in\mathcal{L}^+.
\end{equation*}
The most expensive part of this operation is the matrix-matrix product of $\restr{\mat{X}}{\bs}\in\C^{|\bs|\times d}$ with $\conj{\mV}_b$. In a $\Hm$-matrix-vector product with $\mA^\Hm$, one would have to perform a matrix-vector product with $\conj{\mV}_b$ instead. The other expensive part is a matrix-vector product with $\mU_b$, also done in the regular $\Hm$-matrix-vector product. Thus, in terms of floating point operations, the total application of $\bbC$ has a cost of approximately $(d+1)/2$ matrix-vector products with $\mA^\Hm$. However, the matrix-matrix products can be much more efficient on hardware, reducing the actual relative cost in practice. 

\section{Numerical experiments}\label{sec:6_numexp}

We aim to evaluate the effectiveness of CORK-GMRES for frequency sweeping of exterior scattering problems via Helmholtz BEM. Using multiple examples, the convergence behavior over the frequency range is investigated for both the left- and right-preconditioned methods, and their cost is compared with naively solving systems at individual frequencies. Direct comparisons against other frequency sweeping approaches, as well as tackling high-frequency regimes, are out of the scope of this manuscript.

\subsection{Experimental setup}\label{ssec:numexp-setup}

In the experiments, we consider the exterior Dirichlet problem and use the \emph{indirect combined-field integral equation (CFIE) formulation}~\cite{ColtonKress2019}. The boundary integral equation involves the boundary integral operator
$\mathcal{A}(\kappa) + \frac{1}{2}\mathcal{I}$, with the identity operator $\mathcal{I}$ and kernel function 
\begin{equation*}
    g_\kappa(\px,\py)=\Gk(\px,\py)-i\eta\dfrac{\partial \Gk(\px,\py)}{\partial n_\py}.
\end{equation*} 
We choose $\eta=\kappa$ and construct the $\mathcal{H}$-tensor for $\mathcal{A}$. The discretized identity results in a contribution of a sparse mass matrix $\mat{M}$.\footnote{The mass matrix is appropriately incorporated in all the operations we perform with the system matrix. Its computational contribution is negligible.} The right-hand side $\vb$ contains the Dirichlet boundary values of the excitation. 
The excitation source is either a plane wave with direction $\mathbf{d}$ ($\|\mathbf{d}\|=1$) or a point source located at $\mathbf{p}$, s.t.\ for all $i=1,\dots,n$,
\begin{equation*}
    b_{i}(\kappa) = -\int_\Gamma u^I(\px;\kappa)\psi_i(\px)\mathrm{d}s_\px \;\;\text{with}\;\;
    u^I(\px;\kappa)=\mathrm{exp}(i\kappa\px\cdot\mathbf{d}) \;\;\text{or}\;\; u^I(\px;\kappa)=\Gk(\px,\mathbf{p}).
\end{equation*}

The discretization we employ uses linear discontinuous Galerkin BEM where a mesh of $N_t$ triangles and $N_v$ vertices approximates the boundary. The operators that are discretized have three basis functions for each triangle. Using a sparse matrix $\mat{T}$, these are combined into appropriate continuous linear functions for each of the mesh vertices. This means that the $\mathcal{H}$-tensors and $\mathcal{H}$-matrices we construct are of size $3N_t \times 3N_t$ while the actual system matrix is of size $N_v\times N_v$. Applying it involves only additional cheap operations with $\mat{T}$. See \cite[\S2.2.1,\S5.1.2]{Dirckx2024} for more details.

For all the experiments, we apply the inexact greedy versions of LCORK- and RCORK-GMRES. For the inexactness, we use the bounds from \cref{ssec:inexact-cork} with $\epsilon=10^{-5}$ and unknown smallest singular value replaced by $1$. The relative tolerance used for $\sigma_0$ in LCORK-GMRES is set to $\epsilon/10$. $\sigma_1$ and $\sigma_2$ are fixed at the endpoints of the considered frequency range, and $\sigma_0$ is put at the midpoint. 

The system solution and coefficient matrix application are performed as discussed in \cref{ssec:bem-cork}. In BiCGStab(2), we use $\mat{M}^{\minone}$ as right preconditioner. Unless specified otherwise, the same construction tolerance $\tau=10^{-5}$ is used for $\mathcal{H}$-tensor and direct $\mathcal{H}$-matrix construction, as well as the construction of $\mathcal{H}$-matrices through the $\mathcal{H}$-tensor. To compute the integrals, Sauter quadrature~\cite{SauterSchwab2011book} of degrees $(7,11,13,15)$ are employed. For both AAA approximation and greedy shift selection, we employ 10 samples per dimensionless wavenumber ($\kappa\cdot\mathrm{diam}(\Gamma)$).

All experiments are performed on a compute node consisting of 2 Intel Xeon Platinum 8360Y (Ice Lake) CPUs with 36 cores each, one L3 cache per CPU and 256 GiB total RAM.\footnote{The compute node is part of the `wICE' cluster provided by the VSC (Flemish Supercomputer Center), funded by the Research Foundation -- Flanders (FWO) and the Flemish Government.} Julia code is used for the CORK framework with open-source \CC~package \href{https://gitlab.kuleuven.be/numa/software/beachpack}{BEACHpack} providing efficient implementation of the BEM functionality covered in \cref{sec:5_bem}.\footnote{The version of BEACHpack used to produce the results of the numerical experiments can be found at \href{https://gitlab.kuleuven.be/numa/software/beachpack/-/tree/kb-exp-paper-corkgmres?ref_type=tags}{gitlab.kuleuven.be/numa/software/beachpack/-/tree/kb-exp-paper-corkgmres}. The Julia code is available upon request from the authors.}

\subsection{Performance of greedy CORK-GMRES}

To assess the effectiveness of CORK-GMRES using the greedy shift selection, a first experiment is done on a trefoil knot where the incident field originates from a plane wave going through the hole of the knot. LCORK-GMRES and RCORK-GMRES is applied to the frequency range $\kappa\cdot\mathrm{diam}(\Gamma)\in[10,30]$ for $30$ iterations.

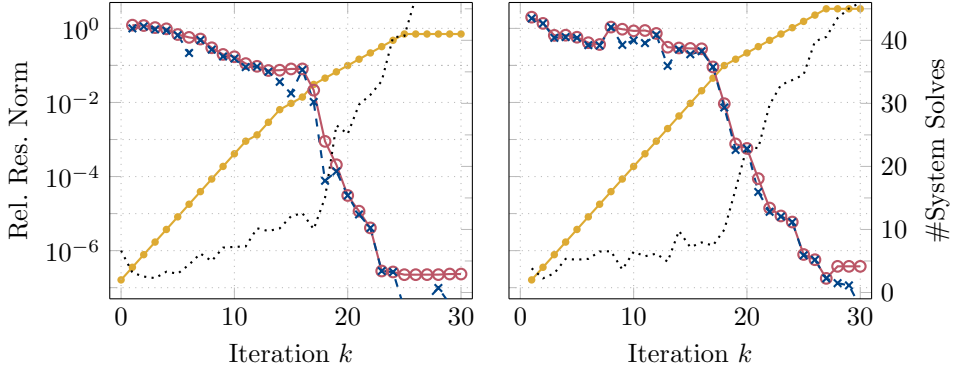
\begin{figure}[tbp]
\centering
\begin{subfigure}[b]{0.49\textwidth}
\begin{tikzpicture}

\pgfplotstableread[col sep=comma]{data/fig2-conv/iglcork-conv-data-trefoil-101612_10.0-30.0-1.0e-5_10.0-10.0-0.001.dat}{\LCorkData}

\begin{axis}[width=\linewidth, 
    axis y line*=right,
    axis x line=none,   
    xmin = -1, xmax = 31,
    ymin = -1, ymax = 46,
    minor ytick={5,15,25,35,45},
    ytick={0,10,20,30,40},
    yticklabels={},
    ymajorgrids=false,
    legend style = {font=\small}]

    \addplot[color=T-Q-HC3,mark=o,mark options=solid,thick] coordinates {(-10,-10)};
    \addplot[color=T-Q-HC4,dashed,mark=x,mark options=solid,thick] coordinates {(-10,-10)};
    \addplot[color=T-Q-HC0,dotted,mark options=solid,thick] coordinates {(-10,-10)};

    \addplot[color=T-Q-HC2,mark=*,mark options=solid,thick, mark size=1pt] table
        [skip first n=1, x index=0, y index = 4]{\LCorkData};

\end{axis}

\begin{semilogyaxis}[width=\linewidth, 
    xmajorgrids = true,
    axis y line* = left,
    ytick pos = left,
    xmin = -1, xmax = 31,
    ymin = 5e-8, ymax = 5e0,
    minor ytick={1e-7,1e-5,1e-3,1e-1,1e-1,1e1},
    ytick={1e-6,1e-4,1e-2,1e-0},
    grid=both, 
    grid style = {dotted,gray},  
    xlabel = {Iteration $k$},
    ylabel = {Rel.\ Res.\ Norm},
    legend style = {font=\small}]

    \addplot[color=T-Q-HC3,mark=o,mark options=solid,thick] table
        [skip first n=1, x index=0, y index = 1]{\LCorkData};
    \addplot[color=T-Q-HC4,dashed,mark=x,mark options=solid,thick] table
        [skip first n=1, x index=0, y index = 2]{\LCorkData};
    \addplot[color=T-Q-HC0,dotted,mark options=solid,thick] table
        [skip first n=1, x index=0, y index = 3]{\LCorkData};

\end{semilogyaxis}

\end{tikzpicture}
\end{subfigure}
\begin{subfigure}[b]{0.49\textwidth}
\begin{tikzpicture}

\pgfplotstableread[col sep=comma]{data/fig2-conv/igrcork-conv-data-trefoil-101612_10.0-30.0-1.0e-5_10.0-10.0-0.001.dat}{\RCorkData}

\begin{axis}[width=\linewidth, 
    axis y line*=right,
    axis x line=none,   
    xmin = -1, xmax = 31,
    ymin = -1, ymax = 46,
    minor ytick={5,15,25,35,45},
    ytick={0,10,20,30,40},
    ylabel={\#System Solves},
    ymajorgrids=false,
    legend style = {font=\small}]

    \addplot[color=T-Q-HC2,mark=*,mark options=solid,thick, mark size=1pt] table
        [skip first n=1, x index=0, y index = 4]{\RCorkData};

\end{axis}

\begin{semilogyaxis}[width=\linewidth, 
    xmajorgrids = true,
    axis y line* = left,
    ytick pos = left,
    xmin = -1, xmax = 31,
    ymin = 5e-8, ymax = 5e0,
    minor ytick={1e-7,1e-5,1e-3,1e-1,1e-1,1e1},
    ytick={1e-6,1e-4,1e-2,1e-0},
    yticklabels={},
    grid=both, 
    grid style = {dotted,gray},  
    xlabel = {Iteration $k$}]

    \addplot[color=T-Q-HC3,mark=o,mark options=solid,thick] table
        [skip first n=1, x index=0, y index = 1]{\RCorkData};
    \addplot[color=T-Q-HC4,dashed,mark=x,mark options=solid,thick] table
        [skip first n=1, x index=0, y index = 2]{\RCorkData};
    \addplot[color=T-Q-HC0,dotted,mark options=solid,thick] table
        [skip first n=1, x index=0, y index = 3]{\RCorkData};
\end{semilogyaxis}

\end{tikzpicture}
\end{subfigure}
\caption{Convergence of LCORK- (left) and RCORK-GMRES (right) applied to the trefoil knot for $\kappa\cdot\mathrm{diam}(\Gamma)\in[10,30]$.
Maximum (solid, blue) and perceived maximum (dashed, red) relative residual norms are shown, as well as the inner relative tolerances (dotted, black) and the cumulative number of BiCGStab(2) solves (dots, yellow).}\label{fig:numexp-conv}
\end{figure}

\Cref{fig:numexp-conv} shows the maximum residual norm over the frequency range, together with the residual norm perceived by the greedy selection (i.e.\ the residual norm at the wavenumber where the estimate is largest). Both methods show good convergence, attaining a relative residual norm below $10^{-5}$ at iteration $22$ (LCORK) and $23$ (RCORK). In addition, the convergence perceived through the greedy shift selection corresponds well to the actual convergence. \Cref{fig:numexp-conv} also reports the cumulative number of BiCGStab(2) solves, showing that $22$ ($23$) iterations require $38$ ($41$) system solutions for LCORK (RCORK). Recall that the number of system solutions can be seen as a proxy for the cost of CORK-GMRES in terms of naive system solutions at individual wavenumbers, although this assumes the system solves encompass the majority of the runtime.

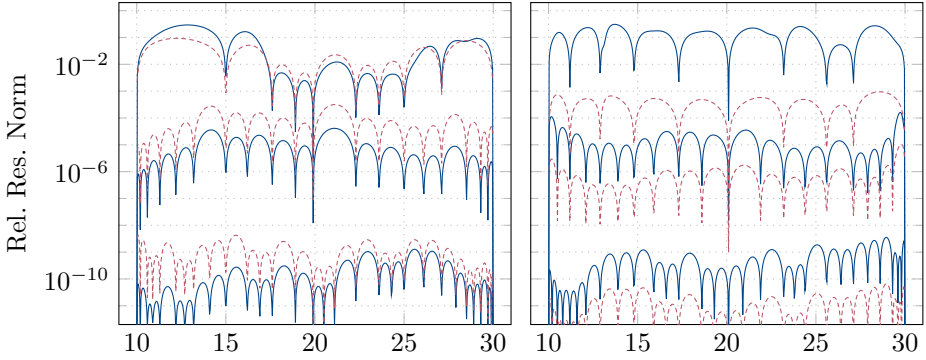
\begin{figure}[tbp]
\centering
\begin{tikzpicture}
\begin{groupplot}[group style={group size=2 by 1, horizontal sep=0.02\linewidth, vertical sep=0.06\linewidth}]

\pgfplotstableread{data/fig3-frequency/iglcork-freq-realres-data_trefoil-knot-10.0-30.0-1.0e-10_10.0-10.0-0.001.dat}{\LeftFreqData}
\pgfplotstableread{data/fig3-frequency/iglcork-freq-compres-data_trefoil-knot-10.0-30.0-1.0e-10_10.0-10.0-0.001.dat}{\LeftEstData}

\pgfplotstableread{data/fig3-frequency/igrcork-freq-realres-data_trefoil-knot-10.0-30.0-1.0e-10_10.0-10.0-0.001.dat}{\RightFreqData}
\pgfplotstableread{data/fig3-frequency/igrcork-freq-compres-data_trefoil-knot-10.0-30.0-1.0e-10_10.0-10.0-0.001.dat}{\RightEstData}

\pgfplotsset{select coords between index/.style 2 args={
    x filter/.code={
        \ifnum\coordindex<#1\def\pgfmathresult{}\fi
        \ifnum\coordindex>#2\def\pgfmathresult{}\fi
    }
}}

\nextgroupplot[width=0.52\linewidth, 
    xmajorgrids = true,
    ymode = log,
    xmin = 9, xmax = 31,
    ymin = 2e-12, ymax = 2e0,
    minor ytick={1e-11,1e-9,1e-8,1e-7,1e-5,1e-4,1e-3,1e-1,1e-0},
    ytick={1e-10,1e-6,1e-2},
    grid=both, 
    grid style = {dotted,gray},
    ylabel = {Rel.\ Res.\ Norm},
    legend to name=grouplegend,
    legend style = {column sep = 1em, font=\scriptsize},
    legend columns = 2]

    \addplot[color=T-Q-HC4] table
        [skip first n=1, x index=0, y index = 1]{\LeftFreqData};
    \addplot[color=T-Q-HC3,dashed, dash pattern=on 1.5pt off 1.5pt, line cap=round] table
        [skip first n=1, x index=0, y index = 1]{\LeftEstData};
        
    \addplot[color=T-Q-HC4] table
        [skip first n=1, x index=0, y index = 2]{\LeftFreqData};
    \addplot[color=T-Q-HC3,dashed, dash pattern=on 1.5pt off 1.5pt, line cap=round] table
        [skip first n=1, x index=0, y index = 2]{\LeftEstData};
        
    \addplot[color=T-Q-HC4] table
        [skip first n=1, x index=0, y index = 3]{\LeftFreqData};
    \addplot[color=T-Q-HC3,dashed, dash pattern=on 1.5pt off 1.5pt, line cap=round] table
        [skip first n=1, x index=0, y index = 3]{\LeftEstData};

\nextgroupplot[width=0.52\linewidth, 
    xmajorgrids = true,
    ymode = log,
    xmin = 9, xmax = 31,
    ymin = 2e-12, ymax = 2e0,
    minor ytick={1e-11,1e-9,1e-8,1e-7,1e-5,1e-4,1e-3,1e-1,1e-0},
    ytick={1e-10,1e-6,1e-2},
    yticklabels={},
    grid=both, 
    grid style = {dotted,gray},
    legend style = {font=\small},
    unbounded coords=jump]

    \addplot[color=T-Q-HC4] table
        [skip first n=1, x index=0, y index = 1]{\RightFreqData};
    \addplot[color=T-Q-HC4] table
        [skip first n=1, x index=0, y index = 2]{\RightFreqData};
    \addplot[color=T-Q-HC4] table
        [skip first n=1, x index=0, y index = 3]{\RightFreqData};

    \addplot[color=T-Q-HC3,dashed, dash pattern=on 1.5pt off 1.5pt, line cap=round] table
        [skip first n=1, x index=0, y index = 1]{\RightEstData};
    \addplot[color=T-Q-HC3,dashed, dash pattern=on 1.5pt off 1.5pt, line cap=round] table
        [skip first n=1, x index=0, y index = 2]{\RightEstData};
    \addplot[color=T-Q-HC3,dashed, dash pattern=on 1.5pt off 1.5pt, line cap=round] table
        [skip first n=1, x index=0, y index = 3]{\RightEstData};

\end{groupplot}
\end{tikzpicture}
\caption{True ($\|\vec{r}(\mu)\| \, /\, \|\vb(\mu)\|$; solid, blue) and computed ($\|\tilde{\bvec{r}}(\mu)\| \, /\, \|\bvd(\mu)\|$; dashed, red) relative residual norms of LCORK- (left) and RCORK-GMRES (right) applied to the trefoil knot, for iteration numbers $10$, $20$ and $30$.
}\label{fig:numexp-conv2}
\end{figure}

In \cref{fig:numexp-conv2}, the residual norms are plotted as a function of the frequency, for iteration numbers $10$, $20$ and $30$ and a stricter inexactness tolerance of $10^{-10}$. The greedy selection is effective in reducing the residual norm across the range.\footnote{The scale of the computed residual norm differs from that of the true one, yet it is the shape over the frequency range that matters. The discrepancy in scale is because the computed residual is that of the pencil, not the actual parametrized system.}

\subsection{Convergence results on multiple meshes}

The robustness of CORK-GMRES is tested by performing 30 iterations on six different shapes with $\kappa\cdot\mathrm{diam}(\Gamma)\in[20,30]$. The number of triangles and vertices range from \num{98304} to \num{195014}, and \num{49154} to \num{97509} respectively. \Cref{tab:meshes} reports on the required CORK-GMRES iterations (Out), the number of required solves, i.e.\ FGMRES iterations (Sol), and the total number of matrix-vector products (MV), to attain a maximal relative residual norm below $\epsilon=10^{-1}$, $10^{-3}$ and $10^{-5}$. The time required to attain $\epsilon=10^{-5}$ is reported as well. 

The results of LCORK- and RCORK-GMRES are comparable, both in terms of BiCGStab(2) solves and matrix-vector products, as well as time spent. The required number of `naive solves' to attain the desired accuracy over the whole wavenumber interval ranges from $17$ to $33$. Note that the inexact Krylov bounds in \cref{sec:4_inexact-rk} are partly heuristic, and the smallest singular value is substituted by $1$. We have observed that some cases require to tighten the inexactness tolerance. Scaling also influences where convergence stagnates, see Remark \ref{rmk:scaling}.

\begin{table}[htbp]
    \centering
    \begin{subtable}[htbp]{\linewidth}
    \centering
    \begin{filecontents*}{data/tab2-meshes/iglcork-table-results.csv}
Shape,OutA,ExpA,SolA,MVA,OutC,ExpC,SolC,MVC,OutE,ExpE,SolE,MVE,Time
sphere,8,5,14,258,17,14,23,411,18,18,24,420,112.590847533
lathe part,5,5,9,261,12,12,16,444,13,13,17,453,157.989923822
submarine,10,10,22,1138,11,11,23,1179,12,12,24,1196,147.296088265
indented cube,13,13,28,764,16,16,31,843,16,16,31,843,136.847267198
crankshaft,8,8,18,758,12,12,23,959,15,15,26,1034,284.927622782
trefoil knot,8,7,16,684,11,11,19,807,13,13,21,873,185.574743684
\end{filecontents*}
\pgfplotstabletypeset[
        empty cells with={–\hspace*{0.1em}},
        col sep=comma,
        columns/Shape/.style    = {string type, column type=>{\em}l, column name=\emph{Shape}},
        columns/OutA/.style     = {column type=r, column name=It},
        columns/SolA/.style     = {column type=r, column name=Sol},
        columns/MVA/.style      = {column type=r, column name=MV},
        columns/OutC/.style     = {column type=r, column name=It},
        columns/SolC/.style     = {column type=r, column name=Sol},
        columns/MVC/.style      = {column type=r, column name=MV},
        columns/OutE/.style     = {column type=r, column name=It},
        columns/SolE/.style     = {column type=r, column name=Sol},
        columns/MVE/.style      = {column type=r, column name=MV},
        columns/Time/.style     = {column type=r, precision=1, fixed zerofill, column name=Time},
        columns = {Shape,OutA,SolA,MVA,OutC,SolC,MVC,OutE,SolE,MVE,Time},
        every head row/.style = {
            before row = {
                \toprule
                & \multicolumn{3}{c}{$\epsilon=10^{-1}$} 
                & \multicolumn{3}{c}{$\epsilon=10^{-3}$}
                & \multicolumn{4}{c}{$\epsilon=10^{-5}$} \\ 
                \cmidrule(lr){2-4} \cmidrule(lr){5-7} \cmidrule(lr){8-11}}, 
            after row=\midrule},
        every last row/.style = {after row=\bottomrule}
    ]{data/tab2-meshes/iglcork-table-results.csv}\caption{Left-preconditioned CORK-GMRES}
\end{subtable}
\\
\begin{subtable}[htbp]{\linewidth}
    \centering
    \begin{filecontents*}{data/tab2-meshes/igrcork-table-results.csv}
Shape,OutA,ExpA,SolA,MVA,OutC,ExpC,SolC,MVC,OutE,ExpE,SolE,MVE,Time
sphere,14,14,22,406,17,17,25,457,21,21,29,501,125.566817768
lathe part,10,8,16,464,13,13,19,547,14,14,20,564,188.071660405
submarine,10,10,20,992,12,12,22,1070,13,13,23,1083,139.64096495
indented cube,14,14,28,780,17,17,31,859,19,19,33,897,152.2848113
crankshaft,7,6,13,549,12,12,21,869,14,14,23,919,269.889060178
trefoil knot,9,9,18,782,11,11,20,860,13,13,22,922,191.466257827
\end{filecontents*}
\pgfplotstabletypeset[
        empty cells with={–\hspace*{0.1em}},
        col sep=comma,
        columns/Shape/.style    = {string type, column type=>{\em}l, column name=\emph{Shape}},
        columns/OutA/.style     = {column type=r, column name=It},
        columns/SolA/.style     = {column type=r, column name=Sol},
        columns/MVA/.style      = {column type=r, column name=MV},
        columns/OutC/.style     = {column type=r, column name=It},
        columns/SolC/.style     = {column type=r, column name=Sol},
        columns/MVC/.style      = {column type=r, column name=MV},
        columns/OutE/.style     = {column type=r, column name=It},
        columns/SolE/.style     = {column type=r, column name=Sol},
        columns/MVE/.style      = {column type=r, column name=MV},
        columns/Time/.style     = {column type=r, precision=1, fixed zerofill, column name=Time},
        columns = {Shape,OutA,SolA,MVA,OutC,SolC,MVC,OutE,SolE,MVE,Time},
        every head row/.style = {
            before row = {
                \toprule
                & \multicolumn{3}{c}{$\epsilon=10^{-1}$} 
                & \multicolumn{3}{c}{$\epsilon=10^{-3}$}
                & \multicolumn{4}{c}{$\epsilon=10^{-5}$} \\ 
                \cmidrule(lr){2-4} \cmidrule(lr){5-7} \cmidrule(lr){8-11}}, 
            after row=\midrule},
        every last row/.style = {after row=\bottomrule}
    ]{data/tab2-meshes/igrcork-table-results.csv}\caption{Right-preconditioned CORK-GMRES}
\end{subtable}\caption{Number of outer CORK iterations (It), solves (Sol) and matrix-vector product (MV) for six different shapes and three achieved values of the real relative residual norm $\epsilon=\max_\mu\|\vec{r}(\mu)\|/\|\vec{b}(\mu)\|$. The time in seconds is given for $\epsilon=10^{-5}$.}\label{tab:meshes}
\end{table}

\subsection{Computational and memory cost of the $\mathcal{H}$-tensor}

The computational and memory cost in using CORK-GMRES to perform frequency sweeping, is largely determined by the $\mathcal{H}$-tensor. $\mathcal{H}$-tensors are constructed for the indented cube, hit by a point source, at three construction tolerances and over four different frequency ranges. \Cref{tab:abs-htensor-cost} compiles the memory footprint, degree and construction time of the $\mathcal{H}$-tensors, as well as the three important $\mathcal{H}$-tensor operations: slicing, matrix-vector multiplication (MV) and coefficient matrix application (CMA). More extensive experiments on the $\mathcal{H}$-tensor representation are performed in \cite{Bruyninckx2026}.

\begin{table}[htbp]
    \centering
    \centering

\begin{filecontents*}{data/tab3-timings/indented-cube-5_htensor-timings.csv}
tol;intervals;Mem;Degr;Build;SliceMean;SliceDev;MV;CMA
\midrule $10^{-3}$;$[10,20]$;34.402201184;8.0;150.889502509;0.3723501615;0.3723501615;0.179755695;0.2164238765
;$[20,30]$;34.895438704;9.0;153.844496704;0.3936454325;0.3936454325;0.1847869025;0.232100244
;$[30,40]$;35.816642304;9.0;162.307291446;0.4145606745;0.4145606745;0.1874486245;0.246731118
;$[10,30]$;41.839564896;12.0;268.387515144;0.463537103;0.463537103;0.19749002;0.2856457385
\midrule $10^{-5}$;$[10,20]$;57.00537448;11.0;583.31392887;1.065270613;1.065270613;0.2852821765;0.4139099885
;$[20,30]$;58.269675904;11.0;593.911764358;1.1446405425;1.1446405425;0.291907246;0.4278983015
;$[30,40]$;59.991074128;11.0;629.782024983;1.234889338;1.234889338;0.308852335;0.439145884
;$[10,30]$;68.014422368;14.0;980.995242404;1.470867067;1.470867067;0.325459258;0.529819221
\midrule $10^{-7}$;$[10,20]$;89.54324408;12.0;1801.691805024;2.6964040545;2.6964040545;0.4624007645;0.62476725
;$[20,30]$;91.753255504;12.0;1859.400688796;2.878023915;2.878023915;0.486256403;0.654712319
;$[30,40]$;94.587928704;12.0;1940.01873924;3.076766874;3.076766874;0.500321633;0.673218024
;$[10,30]$;107.690317328;17.0;2909.387719868;3.5979488535;3.5979488535;0.5507280955;0.888570688
\end{filecontents*}
\pgfplotstabletypeset[
    col sep=semicolon,
  columns/tol/.style   ={string type, column type=>{}c, column name={\small $\tau$}},
  columns/intervals/.style    ={string type, column type=>{}r, column name=$\kappa\cdot\mathrm{diam}(\Gamma)$},
  columns/Mem/.style    ={column type=c, precision=1, fixed zerofill, column name=Mem, 
  postproc cell content/.append style={
    @cell content={\tablenum[round-mode=places, round-precision=1, table-format=3.1]{##1}}
  }},
  columns/Degr/.style  ={column type=c, precision=0, fixed zerofill, column name=$d$,
  postproc cell content/.append style={
    @cell content={\tablenum[round-mode=places, round-precision=0, table-format=2.0]{##1}}
  }},
  columns/Build/.style  ={column type=c, precision=1, fixed zerofill, column name=Build, 
  postproc cell content/.append style={
    @cell content={\tablenum[round-mode=places, round-precision=1, table-format=4.1]{##1}}
  }},
  columns/SliceMean/.style  ={column type=c, precision=3, fixed zerofill, column name=Slice, 
  postproc cell content/.append style={
    @cell content={\tablenum[round-mode=places, round-precision=3, table-format=1.3]{##1}}
  }},
  columns/MV/.style  ={column type=c, precision=3, fixed zerofill, column name=MV, 
  postproc cell content/.append style={
    @cell content={\tablenum[round-mode=places, round-precision=3, table-format=1.3]{##1}}
  }},
  columns/CMA/.style  ={column type=c, precision=3, fixed zerofill, column name=CMA, 
  postproc cell content/.append style={
    @cell content={\tablenum[round-mode=places, round-precision=3, table-format=1.3]{##1}}
  }},
  columns={tol,intervals,Mem,Degr,Build,SliceMean,MV,CMA},
  every head row/.style={before row=\toprule},
  every last row/.style={after row=\bottomrule},
  every row/.style={after row=\midrule},
  assign column name/.style={/pgfplots/table/column name={\small{#1}}}
]{data/tab3-timings/indented-cube-5_htensor-timings.csv}\caption{Memory usage in GB, degree $d$, and construction, slice, matrix-vector product (MV), and coefficient matrix application (CMA) times in seconds for the $\mathcal{H}$-tensor, for the indented cube at three different tolerances $\tau$ and four different intervals.}\label{tab:abs-htensor-cost}
\end{table}

Both widening the wavenumber interval, and moving it to higher frequencies, increase the cost in all aspects, although it is mainly the former that has a significant impact here. Lowering the tolerance also has a large effect, especially on the construction time, because the quadrature degrees are increased as well.\footnote{For tolerances $\tau=10^{-3}$ ($10^{-7}$), quadrature degrees are decreased (increased) by $2$.} Note that the computation time of CMA compared to MV, is lower than the degree would indicate.

In \cref{tab:rel-htensor-cost}, the memory, construction and matrix-vector product of the $\mathcal{H}$-tensor are compared to those of the $\mathcal{H}$-matrices at the frequency boundaries. Memory footprint of the $\mathcal{H}$-tensor is around 2.7--3.5 times greater than a $\mathcal{H}$-matrix. The matrix-vector product shows a similar ratio, as expected. The increase in construction time is more significant, ranging between 13 and 24 times longer. This discrepancy is partly because tensorial ACA first overestimates the ranks of the $\mathcal{H}$-tensor, which is then alleviated by a recompression step. Additionally, whereas memory usage depends only linearly on the rank, the construction also has terms with quadratic and cubic rank dependence. Thus, the rank increase from $\Hm$-matrix to $\Hm$-tensor impacts the construction time more strongly.

\begin{table}[htbp]
    \centering
    \centering

\pgfplotstableread[col sep=semicolon]{data/tab3-timings/indented-cube-5_timing-ratios.csv}\ratiotable

\pgfplotstabletypeset[
    col sep=semicolon,
  columns/tol/.style   ={string type, column type=>{}c, column name={\small $\tau$}},
  columns/intervals/.style    ={string type, column type=>{}r, column name=$\kappa\cdot\mathrm{diam}(\Gamma)$},
  columns/Memm/.style    ={column type=c, column name={}, 
  postproc cell content/.append style={
    @cell content={\tablenum[round-mode=places, round-precision=2, table-format=1.2]{##1}\hspace{-2em}}
  }},
  columns/Memd/.style    ={column type=c, column name={\hspace{-3em}Mem},
  postproc cell content/.append style={
    @cell content={\hspace{-1em}{\,\footnotesize$\pm$}\footnotesize\tablenum[round-mode=places, round-precision=2, table-format=1.2]{##1}}
  }},
  columns/Buildm/.style  ={column type=c, column name={},
  postproc cell content/.append style={
    @cell content={\tablenum[round-mode=places, round-precision=2, table-format=2.2]{##1}}
  }},
    columns/Buildd/.style  ={column type=c, column name={\hspace{-3em}Build},
    postproc cell content/.append style={
    @cell content={\hspace{-1em}{\,\footnotesize$\pm$}\footnotesize\tablenum[round-mode=places, round-precision=2, table-format=1.2]{##1}}
  }},
  columns/MVm/.style  ={column type=c, column name={},
  postproc cell content/.append style={
    @cell content={\tablenum[round-mode=places, round-precision=2, table-format=1.2]{##1}}
  }},
  columns/MVd/.style  ={column type=c, column name={\hspace{-3em}MV},
  postproc cell content/.append style={
    @cell content={\hspace{-1em}{\,\footnotesize$\pm$}\footnotesize\tablenum[round-mode=places, round-precision=2, table-format=1.2]{##1}}
  }},
  columns={tol,intervals,Memm,Memd,Buildm,Buildd,MVm,MVd},
  every head row/.style={before row=\toprule},
  every last row/.style={after row=\bottomrule},
  every row/.style={after row=\midrule},
  assign column name/.style={/pgfplots/table/column name={\small{#1}}}
]{\ratiotable}\caption{Ratios of memory usage, construction time and matrix-vector product timings (MV) of the $\mathcal{H}$-tensor over the $\mathcal{H}$-matrices at the wavenumber interval boundaries, for the indented cube at three different tolerances $\tau$ and four different intervals.}\label{tab:rel-htensor-cost}
\end{table}

Using Tables \ref{tab:meshes}, \ref{tab:abs-htensor-cost} and \ref{tab:rel-htensor-cost}, one is able to more concretely compare the total cost of CORK-GMRES against a naive solve. Let us investigate the frequency sweeping of the indented cube  using RCORK, over $\kappa\cdot\mathrm{diam}(\Gamma)\in[20,30]$ with tolerances of $10^{-5}$. Considering $19$ outer iterations, $33$ solves, and $897$ matrix-vector products, the total cost, neglecting orthogonalization, is: $19$ CMAs and slices, $33$ MVs with the $\mathcal{H}$-tensor, and $897$ MVs with $\mathcal{H}$-matrices. This results in a total time of $137.3$ seconds, which corresponds well with the total time measured in \cref{tab:meshes}. A naive solve, assuming $27$ matrix-vector products are required (In/Sol), would take around $2.94$ seconds. Again, actual measured naive solves take $2.92$--$3.50$ seconds. The estimated number of desired frequencies such that CORK-GMRES is superior, is only around $40$ to $53$.

\subsection{Frequency sweeping acoustic scattering by a submarine}

As a final experiment, we perform a frequency sweep on the larger mesh of a submarine ($N_t=\num{400886}$, $N_v=\num{200445}$).\footnote{Here, we used a VSC `big memory' node to fit the $\Hm$-tensor in memory. The hardware difference is the RAM increase to 2048 GiB and that two NUMA domains are present instead of one.} The source is a plane wave with frequency between $100$ and $250$ Hz. This corresponds to $\kappa\cdot\mathrm{diam}(\Gamma)\in[25.97,64.93]$ or about 4 to 10 wavelengths over the submarine's length. During the greedy shift selection, the true residual norm at the next shift is recorded. Once this norm is below the tolerance three times in a row, we consider CORK-GMRES to have converged, which was the case here.

\Cref{fig:freq-sweep-bars} displays the time spent in the different parts of CORK-GMRES. The orthogonalization cost is negligible and thus not shown. The main cost is in the system solution, around 80\% of the total time. The cost of LCORK- and RCORK-GMRES is similar. To improve greedy shift selection and detection of convergence, one may compute the true residual norm at multiple wavenumbers. Say we test $\ell$ wavenumbers, the cost of greedy shift selection would be roughly $\ell$ times higher.

\begin{figure}
    \centering
    \begin{tikzpicture}[scale=1]

\pgfplotstableread[col sep=comma]{data/fig4-bars/lcork-bar-data-beTSSi_mod2_res16_25.97-64.93_1.0e-5-1.0e-5_0.01-0.01-0.0001.csv}{\leftdataT}
\pgfplotstableread[col sep=comma]{data/fig4-bars/rcork-bar-data-beTSSi_mod2_res16_25.97-64.93_1.0e-5-1.0e-5_0.01-0.01-0.0001.csv}{\rightdataT}

\begin{axis}[bar width = 9pt, ymin=4.5, ymax=21.5 ,xmin=0, xmax=2e3, 
            height=0.45\linewidth, width=0.8\linewidth, 
            xtick={0,500,1000,1500,2000,2500,3000}, 
            xlabel={Time [s]}, 
            minor x tick num=4,
            xmajorgrids, xminorgrids, 
            minor grid style = {dotted},
            y tick label style={major tick length=0pt}, ytick={8.5,15.5}, 
            yticklabels={LCORK, RCORK},
            legend pos = north east,
            legend columns = -1,
            legend style={
                /tikz/every even column/.append style={column sep=5pt},
            }
            ]

\addplot[fill=T-Q-MC2, pattern=crosshatch,
        only marks, mark=square*, mark size = 3] coordinates {(-10,-10)};
\addplot [fill=T-Q-MC3, pattern=north east lines, only marks, mark=square*, mark size = 3] coordinates {(-10,-10)};
\addplot [fill=T-Q-MC4, pattern=crosshatch dots, only marks, mark=square*, mark size = 3] coordinates {(-10,-10)};
\addplot [fill=T-Q-MC5, pattern=north west lines, only marks, mark=square*, mark size = 3] coordinates {(-10,-10)};
\addplot [fill=T-Q-MC6, fill=white, only marks, mark=square*, mark size = 3] coordinates {(-10,-10)};

\addplot[xbar interval, fill=T-Q-MC2] 
    table[x index=0, y expr={6+\coordindex}]{\leftdataT};
\addplot[xbar interval, pattern=crosshatch, ]
    table[x index=0, y expr={6+\coordindex}]{\leftdataT};
    
\addplot [xbar interval, fill=T-Q-MC3] 
    table[x index=1, y expr={7+\coordindex}]{\leftdataT};
\addplot [xbar interval, pattern=north east lines, ]
    table[x index=1, y expr={7+\coordindex}]{\leftdataT};

\addplot [xbar interval, fill=T-Q-MC4] 
    table[x index=2, y expr={8+\coordindex}]{\leftdataT};
\addplot [xbar interval, pattern=crosshatch dots, ]
    table[x index=2, y expr={8+\coordindex}]{\leftdataT};

\addplot [xbar interval, fill=T-Q-MC5] 
    table[x index=3, y expr={9+\coordindex}]{\leftdataT};
\addplot [xbar interval, pattern=north west lines, ]
    table[x index=3, y expr={9+\coordindex}]{\leftdataT};

\addplot [xbar interval, fill=T-Q-MC6] 
    table[x index=5, y expr={10+\coordindex}]{\leftdataT};

\addplot[xbar interval, fill=T-Q-MC2] 
    table[x index=0, y expr={13+\coordindex}]{\rightdataT};
\addplot[xbar interval, pattern=crosshatch, ]
    table[x index=0, y expr={13+\coordindex}]{\rightdataT};

\addplot [xbar interval, fill=T-Q-MC3] 
    table[x index=1, y expr={14+\coordindex}]{\rightdataT};
\addplot [xbar interval, pattern=north east lines, ]
    table[x index=1, y expr={14+\coordindex}]{\rightdataT};

\addplot [xbar interval, fill=T-Q-MC4] 
    table[x index=2, y expr={15+\coordindex}]{\rightdataT};
\addplot [xbar interval, pattern={crosshatch dots[radius=2pt, distance=10pt]}, pattern color=black]
    table[x index=2, y expr={15+\coordindex}]{\rightdataT};

\addplot [xbar interval, fill=T-Q-MC5] 
    table[x index=3, y expr={16+\coordindex}]{\rightdataT};
\addplot [xbar interval, pattern=north west lines, ]
    table[x index=3, y expr={16+\coordindex}]{\rightdataT};

\addplot [xbar interval, fill=T-Q-MC6] 
    table[x index=5, y expr={17+\coordindex}]{\rightdataT};

\legend{CMA,Slice,Solve,Greedy,Total}

\end{axis}
\end{tikzpicture}
    \caption{Time spent on the different components of LCORK- and RCORK-GMRES when applied to the refined submarine: Coefficient matrix application (CMA), slicing, system solution, and greedy shift selection. Total time is also shown.}
    \label{fig:freq-sweep-bars}
\end{figure}
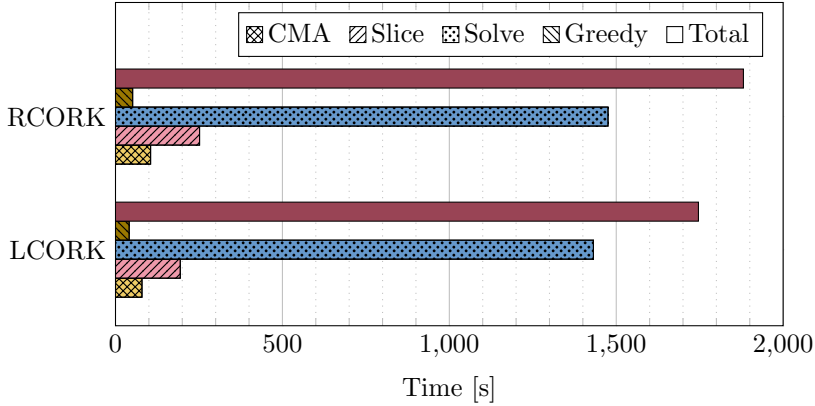 

Finally, we compare LCORK-GMRES to naively solving at a set of wavenumber samples in the range. \Cref{fig:freq-sweep-comparison} plots the cost in time as a function of the number of samples, including construction costs. For LCORK-GMRES, the time also considers the computation of the approximate solution $\vx_m(\kappa)$ for each sample. There are two types of naive solves: constructing a $\Hm$-matrix at each sample directly, or constructing one $\Hm$-tensor and slicing it at each sample.

First, we mention that on average, the construction of the $\Hm$-tensor took approximately $52.2$ times as long as the $\Hm$-matrix construction at the middle frequency while the ratio in memory footprint is only around $3.35$. The discrepancy is larger than in \cref{tab:rel-htensor-cost}, which is explained by worse load balancing. It is observed that for the $\Hm$-tensor, construction of the largest blocks takes longer, relative to the small blocks, than for the $\Hm$-matrices. If a few blocks, less than the number of cores, take most of the time, load balancing degrades unless parallelization is done within those blocks. Further investigation, e.g.\ comparing to sequential construction, is necessary.

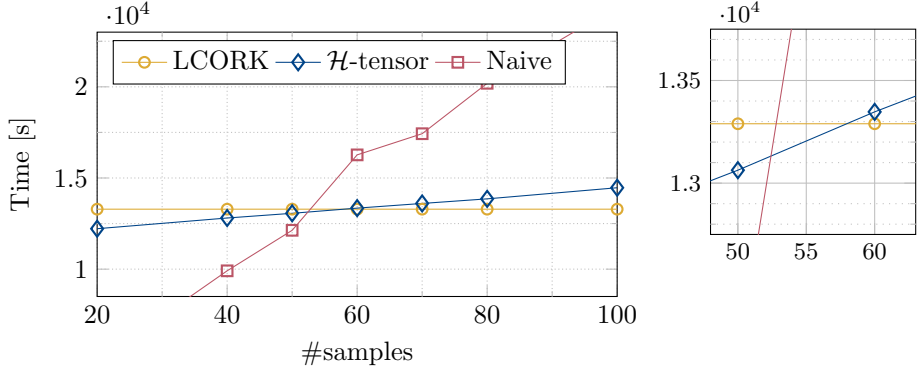
\begin{figure}
    \centering
    \begin{subfigure}[t]{0.65\textwidth}
\begin{tikzpicture}[baseline=(current bounding box.north),scale=1,          
                    spy using outlines={
                        rectangle,
                        magnification=3,
                        width=5cm,
                        height=3cm,
                        connect spies
                    }]

\pgfplotstableread[col sep=comma]{data/fig5-comparison/lcork-comparison-data-beTSSi_mod2_res16_26.0-64.9_1.0e-5-1.0e-5_0.01-0.01-0.0001.csv}{\leftdata}
\pgfplotstableread[col sep=comma]{data/fig5-comparison/rcork-comparison-data-beTSSi_mod2_res16_26.0-64.9_1.0e-5-1.0e-5_0.01-0.01-0.0001.csv}{\rightdata}
\pgfplotstableread[col sep=comma]{data/fig5-comparison/naive-data_beTSSi_mod2_res16-26.0-64.9-1.0e-5.csv}{\naivedata}
\pgfplotstableread[col sep=comma]{data/fig5-comparison/htensor-data_beTSSi_mod2_res16-26.0-64.9-1.0e-5-rhs.csv}{\htensordata}

\begin{axis}[height=0.6\linewidth, width=\linewidth, 
            xmin = 20, xmax = 100,
            ymin = 8.5e3, ymax = 2.3e4,
            xlabel={\#samples}, 
            ylabel={Time [s]},
            xtick={20,40,60,80,100},
            ytick={1e4,1.5e4,2e4,3e4,4e4,5e4,6e4,7e4,8e4,9e4,1e5}, 
            extra y ticks={7.5e3,1.25e4, 1.75e4, 2.25e4},
            grid=both,
            major grid style={densely dotted},
            minor grid style={densely dotted},
            extra y tick label=\empty,
            extra y tick style={
                major tick length=2pt,              
                major grid style={densely dotted}
            },
            extra x ticks={30,50,70,90},
            extra x tick label=\empty,
            extra x tick style={
                major tick length=2pt,              
                major grid style={densely dotted}
            },
            legend pos = north west,
            legend columns = -1,
            legend style={anchor = north west}
            ]

    \addplot[color=T-Q-HC2,mark=o,mark options={solid,thick}] table
        [skip first n=1, x index=0, y index = 4]{\leftdata};

    \addplot[color=T-Q-HC4,mark=diamond, mark options={solid,thick}, mark size=3pt] table
        [skip first n=1, x index=0, y index = 3]{\htensordata};

    \addplot[color=T-Q-HC3,mark=square,mark options={solid,thick}] table
        [skip first n=1, x index=0, y index = 1]{\naivedata};

    \legend{LCORK, $\Hm$-tensor, Naive};

\coordinate (magnifying glass) at (axis cs: 120, 1.4e4);
\coordinate (spy point) at (axis cs: 60, 1.5e4);

\end{axis}

\end{tikzpicture}
\end{subfigure}
\begin{subfigure}[t]{0.33\textwidth}
\begin{tikzpicture}[baseline=(current bounding box.north),scale=1]

\pgfplotstableread[col sep=comma]{data/fig5-comparison/lcork-comparison-data-beTSSi_mod2_res16_26.0-64.9_1.0e-5-1.0e-5_0.01-0.01-0.0001.csv}{\leftdata}
\pgfplotstableread[col sep=comma]{data/fig5-comparison/rcork-comparison-data-beTSSi_mod2_res16_26.0-64.9_1.0e-5-1.0e-5_0.01-0.01-0.0001.csv}{\rightdata}
\pgfplotstableread[col sep=comma]{data/fig5-comparison/naive-data_beTSSi_mod2_res16-26.0-64.9-1.0e-5.csv}{\naivedata}
\pgfplotstableread[col sep=comma]{data/fig5-comparison/htensor-data_beTSSi_mod2_res16-26.0-64.9-1.0e-5-rhs.csv}{\htensordata}

\begin{axis}[height=\linewidth, width=\linewidth, 
            xmin = 48, xmax = 63,
            ymin = 1.275e4, ymax = 1.375e4,
            restrict x to domain=40:80,
            restrict y to domain=1e4:2e4,
            xtick={50,55,60,65}, 
            xlabel={}, 
            ylabel={},
            ytick={1.25e4,1.3e4,1.35e4,1.4e4}, 
            tick label style={font=\small},
            extra y ticks={1.26e4, 1.27e4, 1.28e4, 1.29e4, 1.31e4, 1.32e4, 1.33e4, 1.34e4, 1.36e4, 1.37e4, 1.38e4, 1.39e4},
            grid=both,
            major grid style={solid},
            minor grid style={dotted},
            extra y tick label=\empty,
            extra y tick style={
                major tick length=2pt,              
                major grid style={dotted}
            },
            ]

    \addplot[color=T-Q-HC2,mark=o,mark options={solid,thick}] table
        [skip first n=1, x index=0, y index = 4]{\leftdata};

    \addplot[color=T-Q-HC4,mark=diamond, mark options={solid,thick}, mark size=3pt] table
        [skip first n=1, x index=0, y index = 3]{\htensordata};

    \addplot[color=T-Q-HC3,mark=square,mark options={solid,thick}] table
        [skip first n=1, x index=0, y index = 1]{\naivedata};

\end{axis}

\end{tikzpicture}
\end{subfigure}
    \caption{Frequency sweeping cost in time as a function of the number of samples, for LCORK-GMRES, and naive solution with and without $\Hm$-tensor; applied to the refined submarine.}
    \label{fig:freq-sweep-comparison}
\end{figure}

The ratio in construction time influences the results in \cref{fig:freq-sweep-comparison}. Using the $\Hm$-tensor, either with CORK-GMRES or without, pays off once over 50 to 55 samples. If CORK-GMRES is used, almost all the cost is `offline', because computing the approximate solutions is very cheap. Thus, it remains feasible to evaluate the solution at an increasingly large number of frequencies, in contrast to using only the $\Hm$-tensor.

Finally, we note that, in this experiment, the additional cost of CORK-GMRES is low relative to the $\Hm$-tensor construction, partially due to suboptimal parallelization. However, other parameters also affect this ratio, e.g., lowering the quadrature degrees reduces construction cost. Therefore, the relative cost between $\Hm$-tensor construction and CORK-GMRES execution may vary across setups. Frequency sweeping with the combination stays efficient regardless, given that the cost of repeated system solution depends on the parameter setup in a similar fashion.

\section{Conclusions and future work}\label{sec:7_conclusion}

We have developed a new left- and right-preconditioned Krylov method for solving parametrized linear systems through linearization. While this results in very large systems, using the CORK framework the structure within the companion pencil is exploited and efficiency is maintained. The methods are more general than existing linearization methods and Krylov methods for shifted systems. Moreover, convergence is improved through the use of multiple shifts, which can be chosen greedily and fully adaptively in a cheap manner. Incorporating inexact Krylov theory reduces computational cost, since the tolerances in later iterations can be loosened while maintaining accuracy of the overall solution. Finally, we adapt the methodology to allow for  parameter dependence in the right-hand side.

In combination with the $\mathcal{H}$-tensor format, the framework enables efficient frequency sweeping of three-dimensional scattering problems through Helmholtz BEM. It performs well for many shapes and moderate frequencies, and already outperforms the naive approach computationally once a modest number of frequency samples are desired, albeit with increased memory usage.

Future work may include exploration of broader applications, model order reduction (e.g.\ including field evaluation in Helmholtz BEM), development of compatible data-sparse representations for high-frequency scattering problems (for example inspired by \cite{GopalMartinsson2023}), the computation of scattering poles (i.e.\ eigenvalues of integral operator $\mathcal{A}(\kappa)$), and further investigation of the effects of scaling in linearization.

\bibliographystyle{abbrv}
\bibliography{references}

\end{document}